\documentclass{article}

\usepackage{arxiv}

\usepackage[utf8]{inputenc} % allow utf-8 input

\usepackage[english]{babel}

\usepackage[T1]{fontenc}    % use 8-bit T1 fonts
\usepackage{hyperref}       % hyperlinks
\usepackage{url}            % simple URL typesetting
\usepackage{booktabs}       % professional-quality tables
\usepackage{amsfonts}       % blackboard math symbols
\usepackage{nicefrac}       % compact symbols for 1/2, etc.
\usepackage{microtype}      % microtypography

\usepackage[tbtags]{amsmath}
\usepackage{amsfonts,amssymb}
\usepackage{mathrsfs}
\usepackage{booktabs}
\usepackage{amsthm}
\newtheorem{theorem}{Theorem}
\newtheorem{Lemma}{Lemma}
\newtheorem{remark}{Remark}
\newtheorem{corollary}{Corollary}
\usepackage{relsize}

\title{Some properties of dihedral group codes}

\date{ }
\author{
  Kirill V.~Vedenev \\
  Department of Algebra and Discrete Mathematics\\
  Southern Federal University\\
  Rostov-on-Don, Russia \\
  \texttt{vedenevk@gmail.com} \\
   \And
  Vladimir M.~Deundyak \\
  Department of Algebra and Discrete Mathematics\\
  Southern Federal University \\
  Rostov-on-Don, Russia, \\
  Research Institute <<Specvuzavtomatika>> \\
  Rostov-on-Don, Russia \\
  \texttt{vl.deundyak@gmail.com} \\
}

\begin{document}
\maketitle
\begin{abstract}
In this paper, we study the dihedral codes, i.e. the left ideals of $\mathbb{F}_qD_{n}$ in the case $\gcd(q, n) = 1$.
An explicit algebraic description of the dihedral codes and their duals is obtained. 
In addition, a criterion for self-duality of a dihedral code is obtained.
Bases, generating and check matrices of dihedral codes are constructed. 
Given a dihedral code, we consider exterior and interior codes that are induced by cyclic codes. Using this codes, some properties of generating matrices are described and a connection to cyclic code theory is established.
In addition, some estimates of code parameters are obtained and several illustrative examples are given. 
\end{abstract}



\section*{Introduction}
Let $G$ be a finite group, and let $\mathbb{F}$ be a field. The $\mathbb{F}$--vector space $ \mathbb{F}G := \left\{ \sum_{g \in G} \alpha_g g \mid \alpha_g \in \mathbb{F} \right\}$ with the basis $\{ g \in G \}$, equipped with multiplication defined as follows
\[ 
\left( \sum_{g \in G} a_g g \right) \left(\sum_{g \in G} b_g g \right) = \sum_{g \in G} \left(\sum_{h \in G} a_h b_{h^{-1}g} \right) g,
\]
is called the \textit{group algebra of $G$ over $\mathbb{F}$} (\cite{MilSeh}, p. 131).

Let $\mathbb{F}_q$ be a finite field of order $q$. Any one--sided ideal $I$ of $\mathbb{F}_qG$ is called a \textit{group code} or \textit{$G$--code}. This algebraic approach to study linear codes, introduced independently by S.~Berman in \cite{berman1967theory} and F.~MacWilliams in \cite{macwilliams1970binary}, allows one to build new classes of linear codes and to study some properties of existing ones using algebraic methods (see \cite{zimmerman, borello2020group}, surveys \cite{kelarev2001error, Mil19surv}). In particular, additional algebraic structure can provide efficient coding and decoding algorithms (\cite{DeuKos15, Deundyak2018, deundyak2020graph}). 

Let $u = \sum_{g \in G} u_g g \in \mathbb{F}_qG$. The set $\mathtt{supp}\left( u \right):=\{ g \in G \mid u_g \neq 0 \}$ is called the \textit{support} of $u$ and $w(u) = |\mathtt{supp}\left( u \right)|$ is called the \textit{Hamming weight} of $u$. The minimum distance of a $G$--code $C$ is defined as
\[ d(C) := \min_{c \in C, \; c \neq 0} w(c). \]
Recall that a code $C \subset \mathbb{F}_qG$ is called a $[n,k,d]$--code if $n=|G|$, $k = \dim_{\mathbb{F}_q}(C)$, $d = d(C)$. %As is well--known, if $d \geq 2t + 1$, then the code can correct up to $t$ errors.

A $G$--code $C$ is called \textit{abelian} if $G$ is abelian, otherwise $C$ is called \textit{non--abelian}. Note that the antiautomorphism $x \mapsto x^*$ of $\mathbb{F}G$, given by
\begin{equation} \label{eq:anti}
	\left( \sum_{h \in G} x_h h \right)^{*} := \sum_{h \in G} x_{h^{-1}} h,
\end{equation}
establishes one--to--one correspondence between the left and the right ideals of $\mathbb{F}_qG$, thus it is enough to consider the $G$--codes as the left ideals. 

Many classical codes such as cyclic codes, Reed-Solomon codes and Reed-Muller codes are known to be abelian group codes (see \cite{kelarev2001error, CouGonMar12}). But since non-abelian group algebras have much richer algebraic structure, non--abelian group codes are of particular interest. Non--abelian codes could have very good parameters and proprieties. For example, in \cite{Mil19} some metacyclic group codes with good parameters were constructed and in \cite{Cao} some good self--dual dihedral codes were described. A survey on recent results can be found in \cite{Mil19surv, gao2020lcd}.

In addition, non--abelian codes could possibly  be used to create post--quantum public--key encryption protocols \cite{Bernstein2017, Sendrier}. The first code--based encryption protocol
was developed by R. McEliece in 1978 \cite{MC}. It is based on the use of binary Goppa codes and remains to be secure up to now. The main drawback of the original McEliece cryptosystem is large public key size. In order to overcome it, there were attempts to use other classes of codes, including some of well--known abelian group codes (e.g. Reed--Solomon and Reed--Muller codes), but the most of these modifications were proven to be less secure (see survey in \cite{DeuKos19}). We note that non--abelian groups are a rich source of new classes of promising linear codes to use in code--based cryptography. In addition, non--commutativity could possibly improve the security of code--based encryption protocols. Non--abelian group codes that have some good properties could also be an option to develop secret sharing and secure multiparty computation protocols (\cite{cascudo2018squares}).

By $D_{n}=\langle a, b \mid a^n=1,\: b^2=1,\: bab = a^{-1} \rangle$, $n \geq 2$, we denote the dihedral group of order $2n$.  In \cite{Mar15} the Wedderburn decomposition of the algebra $\mathbb{F}_qD_{n}$, $\gcd(q, 2n)=1$, was described. Using this decomposition, in \cite{VedDeu18}  the authors have obtained algebraic description of all $D_{n}$--codes over $\mathbb{F}_q$ in the case $\gcd(q,2n)=1$. In \cite{VedDeu20} the connection between the codes in $\mathbb{F}_qD_{n}$ and the idempotents of $\mathbb{F}_qD_{n}$ was established in the case $\gcd(q,2n)=1$; also the inverse Wedderburn decomposition isomorphism was explicitly described. In addition, in \cite{VedDeu20} these results were used to study the induced dihedral codes. 
In \cite{VedDeu19} the non--semisimple group algebra over $D_{n} \times D_{m}$ was considered, under certain conditions, the structure of this algebra was described and a generalization of Wedderburn decomposition was obtained; these results were used to describe the codes in this algebra.

In this paper, we continue studying dihedral codes and their properties in the more general case $\gcd(q, n) = 1$. The paper is organized as follows. 
In Section \ref{sec:prel}, some preliminaries on the structure of $\mathbb{F}_qD_n$ and on algebraic description of $D_n$--codes are given. Moreover, some results previously proved in \cite{VedDeu18, VedDeu20} in the case $\gcd(q, 2n) = 1$ are generalized to the case $\gcd(q, n) = 1$. In Section \ref{sec:dual}, for a given $D_n$--code the explicit algebraic description of its dual code is obtained. In addition, the self--dual $D_n$--codes are descried.
In Section \ref{sec:matr}, bases, generating and check matrices of any $D_n$--code are explicitly calculated.
In Section \ref{sec:ind}, for a $D_n$--code $C$ we introduce and study its exterior and interior induced codes $C_{ext}$ and $C_{int}$. The results of this section allows us to study some properties of dihedral codes using cyclic code theory.
In Section \ref{sec:est}, we provide several estimates on code parameters and consider several illustrative examples.
We believe that the results obtained in Sections \ref{sec:matr}---\ref{sec:est} could be used to study the security of code cryptosystems based on $D_n$--codes.

\section{Preliminaries} \label{sec:prel}
\textbf{Dihedral group algebra.}
Since $D_{n}=\langle a, b \mid a^n=1,\: b^2=1,\: bab = a^{-1} \rangle$, it follows that $a^{i}b = ba^{-i}$ for any $i$ and
\begin{equation} \label{eq:dih}
	D_{n}=
	\{e, a, a^2, ..., a^{n-1}, b, ba, ba^2,... , ba^{n-1} \} = 
	\{e, a, a^2, ..., a^{n-1}, b, a^{n-1}b, a^{n-2}b,... , ab \}.
\end{equation}
Eq. (\ref{eq:dih}) implies that any element $u \in \mathbb{F}_qD_{n}$ can be represented as follows:
\begin{equation} \label{eq:pq}
u = P(a) + bQ(a) = P(a) + Q(a^{-1})b,
\end{equation}
where $P$ and $Q$ are polynomials over $\mathbb{F}_q$ of degree less than $n$.	

In \cite{Mar15} in the case $\gcd(q,2n)=1$ the Wedderburn decomposition of $\mathbb{F}_qD_{n}$  was obtained. Below, we consider its generalization in the case when \textit{ $\gcd(q, n) = 1$.}

\textit{Hereinafter we assume that  $\mathrm{gcd}(n,q) = 1$, i.e. $n$ and $q$ are coprime.} 

For a polynomial $g \in \mathbb{F}_q[x]$ such that $g(0) \neq 0$, the reciprocal polynomial is defined as $g^{*}(x) = x^{\mathrm{deg}(g)} g(x^{-1})$. A polynomial $g$ is called auto--reciprocal if $g(x)$ and $g^*(x)$ have the same roots in its splitting field.

As is well--known, the polynomial $x^n-1 \in \mathbb{F}_q[x]$ can be split into monic irreducible factors over $\mathbb{F}_q$;  following \cite{Mar15}, p. 205, we write this decomposition as follows:
\begin{equation}
\label{eq:dec} x^n-1 = (f_1 f_2 \dots f_r) (f_{r+1} f^{*}_{r+1} f_{r+2} f^{*}_{r+2} \dots f_{r+s} f^{*}_{r+s}),
\end{equation}
where $f_1 = x-1$, $f^{*}_j = f_j$ for $1 < j \leq r$, and $f_2=x+1$ if $n$ is even. Here $r$ denotes the number of auto--reciprocal factors and $2s$ denotes the number of non--auto--reciprocal factors.

By $\alpha_j$ we denote a root of $f_j$ from (\ref{eq:dec}). Let
$$\zeta(n) :=
\begin{cases}
1, & \text{$n$ is odd}, \\
2, & \text{$n$ is even}.
\end{cases}$$
Let $\mathrm{M}_2[\mathbb{F}]$ be the algebra of $(2 \times 2)$ matrices over a field $\mathbb{F}$, and let $\langle h \mid h^2 \rangle$ be the cyclic group of order $2$.
Let us define $\mathbb{F}_q$--algebras homomorphisms $\tau_j$ (see \cite{Mar15}) and $\gamma_j$ as follows:
\begin{enumerate} 
	\item $\gamma_1: \mathbb{F}_qD_{n} \rightarrow 
	\mathbb{F}_q \langle h \mid h^2 = e \rangle$, where $\gamma_1 \left( P(a) + bQ(a) \right) = P(1) + Q(1)h$;
	
	\item $\gamma_2: \mathbb{F}_qD_{n} \rightarrow 
	\mathbb{F}_q \langle h \mid h^2 = e \rangle$, where $\gamma_2 \left( P(a) + bQ(a) \right) = P(-1) + Q(-1)h$ for even $n$;
	
	\item $\tau_j: \mathbb{F}_qD_{n} \rightarrow \mathrm{M}_2(\mathbb{F}_q[\alpha_j])$, where $\tau_j \left( P(a) + bQ(a) \right) = \begin{pmatrix} P(\alpha_j) & Q(\alpha_j^{-1}) \\ Q(\alpha_j) & P(\alpha_j^{-1}) \end{pmatrix}$, $j \geq \zeta(n) + 1$.
\end{enumerate}	
For $\zeta(n) + 1 \leq j \leq r$ consider the automorphisms $\sigma_j$ of the algebras $\mathrm{M}_2(\mathbb{F}_q[\alpha_j])$ given by
\begin{equation} \label{eq:zj}
\sigma_j(X) := Z_j^{-1} X Z_j, \quad Z_j := \begin{pmatrix} 1 &  -\alpha_j \\ 1 & -\alpha_j^{-1} \end{pmatrix}.
\end{equation}
As proved in \cite{Mar15}, $\sigma_j (\mathrm{im} (\tau_j)) = \mathrm{M}_2(\mathbb{F}_q[\alpha_j + \alpha_j^{-1}])$ and for $\zeta(n) + 1 \leq j  \leq r$ we have
\begin{equation} \label{eq:dim} 
\dim_{\mathbb{F}_q}(\mathbb{F}_q[\alpha_j + \alpha_j^{-1}]) = {\deg(f_j)}/{2}.
\end{equation}

\begin{theorem}[\cite{Mar15}, Theorem 3.1, Remark 3.4] \label{theor:decomp}
	Let $\gcd(q, n)=1$. Then the following algebra isomorphism holds:
	\begin{equation} \label{eq:2}
	\mathcal{P} = \bigoplus_{j=1}^{r+s} \mathcal{P}_j: \; \mathbb{F}_q D_{n} \rightarrow \bigoplus_{j=1}^{r+s} A_j,
	\end{equation}	
	where
	$$
	\mathcal{P}_j := \begin{cases}
	\gamma_j, & 1 \leq j \leq \zeta(n) \\
	\sigma_j \tau_j, & \zeta(n)+1 \leq j \leq r \\ \tau_j,& r+1 \leq j \leq r+s
	\end{cases}, \quad	
	A_j := \begin{cases}
	\mathbb{F}_q\langle h \mid h^2 \rangle, & 1 \leq j \leq \zeta(n)  \\
	\mathrm{M}_2 (\mathbb{F}_q[\alpha_j + \alpha_j^{-1}]), & \zeta(n) + 1 \leq j \leq r \\	
	\mathrm{M}_2 (\mathbb{F}_q[\alpha_j]), & r + 1 \leq j \leq r+s
	\end{cases}.
	$$			
\end{theorem}

\textbf{Construction of $\mathcal{P}^{-1}$.} To study dihedral codes we also need the explicit construction of $\mathcal{P}^{-1}$. In \cite{VedDeu20} it was obtained in the case when $\gcd(q, 2n) = 1$. Below, we describe $\mathcal{P}^{-1}$ in the case  $\gcd(q, n) = 1$.

Recall that an element $i$ in a ring $R$ is called \textit{an idempotent} if $i^2= i$. In \cite{Mar15} it was proved that for any monic divisor $g(x)$ of the polynomial $x^n-1$, the element
\begin{equation} \label{eq:ef} e_g(x) := -\frac{[(g(x)^*)']^*}{n} \cdot \frac{x^n-1}{g(x)} \end{equation}
is an idempotent of the ring $\mathcal{R}_n := \mathbb{F}_q[x] / \langle x^n-1 \rangle$, and it generates the left ideal $\mathcal{R}_n e_g(x) = \mathcal{R}_n f(x)$, where $f(x)$ stands for ${(x^n - 1)}/{g(x)}$. 	
Note that for any $P(x) \in \mathbb{F}_q[x]$ the expression $(P'(x))^*$ stands for $x^{\deg(P) - 1} P'\left( {1}/{x} \right)$.

Lemma 2 from \cite{VedDeu20} yields that for any root $\alpha$ of $x^n-1$ we have
\begin{equation} \label{eq:orth}
e_{g}(\alpha) = \begin{cases}
1, & g(\alpha) = 0 \\
0, & f(\alpha) = 0
\end{cases}.
\end{equation}

Let $\langle h \mid h^n \rangle$ be the cyclic group of order $n$. Since the map
$\Phi: \mathbb{F}_q \langle h \mid h^n \rangle \rightarrow \mathcal R_n$,
$\Phi(h) = [x]$,
is a $\mathbb{F}_q$--algebra isomorphism, it follows that $e_{g}(a) \in \mathbb{F}_qD_{n}$ are idempotents of $\mathbb{F}_qD_{n}$.

By $\mathrm{id}_S$ we denote the identity map on a set $S$. For $\zeta(n) + 1 \leq j \leq r$ let
\begin{equation} \label{eq:fg}
(x-1) F_{j,1}(x) + f_j(x)F_{j,2}(x) = 1, \quad
(x+1) G_{j,1}(x) + f_j(x)G_{j, 2}(x) = 1
\end{equation}
be the Bezout relations for $\gcd(f_j(x), x-1)=1$ and $\gcd(f_j(x), x+1)=1$, respectively.

\begin{theorem} \label{theor:3}
	Let \( \gcd(q, n) = 1 \), and let the maps $\epsilon_j: A_j \rightarrow \mathbb{F}_qD_{n}$ be defined as
\begin{enumerate}
	\item[1)] for $1 \leq j \leq \zeta(n)$
	\[ \epsilon_j: w_1 + w_2h \mapsto (w_1 + w_2b) e_{f_j}(a), \quad w_1, w_2 \in \mathbb{F}_q \]	
	\item[2)] for $\zeta(n) + 1 \leq j \leq r$ and $M, N, P, Q \in \mathbb{F}_q[x]$:
	\[ 
		\epsilon_j:
		\begin{pmatrix}
		M(\alpha_j + \alpha_j^{-1}) & P(\alpha_j + \alpha_j^{-1}) \\
		N(\alpha_j + \alpha_j^{-1}) & Q(\alpha_j + \alpha_j^{-1}) 
		\end{pmatrix} \mapsto
		(T_1(a) + bT_2(a))e_{f_j}(a),
	\]
	where
	\[ T_1(a) = F_{j, 1}(a)G_{1, j}(a)\left[ -M(a+a^{-1}) + a N(a + a^{-1}) - aP(a + a^{-1}) + a^2Q(a+a^{-1}) \right], \]
	\[ T_2(a) = F_{j, 1}(a)G_{1, j}(a)\left[ -M(a+a^{-1}) + a^{-1}N(a + a^{-1}) - aP(a+a^{-1}) + Q(a+a^{-1})  \right], \]
	\item[3)] for $r+1 \leq j \leq r+s$ and $M, N, P, Q \in \mathbb{F}_q[x]$
	\[ \epsilon_j: \begin{pmatrix}
	M(\alpha_j) & P(\alpha_j^{-1}) \\ N(\alpha_j) & Q(\alpha_j^{-1})
	\end{pmatrix} \mapsto
	[M(a) + bN(a)] e_{f_j}(a) + [Q(a) + bP(a)] e_{f^*_j}(a). \]
\end{enumerate}
	where polynomial representations of elements of field extensions are used in \textit{ 2), 3) }. Then $\mathcal{P}^{-1} = \sum_{j=1}^{r+s} \epsilon_j$ and the maps $\epsilon_j$ are $\mathbb{F}_q$--algebras monomorphisms such that
		\begin{eqnarray} \label{eq:eps}
		\mathcal P_j \epsilon_j = \mathrm{id}_{A_j}, \quad \forall i \in \{ 1 \dots r+s \} \;  \left( i \neq j \right) \Rightarrow \mathcal P_i \epsilon_j = 0.
		\end{eqnarray}
\end{theorem}
\begin{proof}
	Below, we use defenitions of $\mathcal{P}, \mathcal{P}_j, \gamma_j, \tau_j, \sigma_j$ from (\ref{eq:2}). Note that (\ref{eq:2}) and (\ref{eq:orth}) yield $\mathcal{P}_j \epsilon_i = 0$ for $i \neq j$. We will prove that $\mathcal{P}_j \epsilon_j$ is the identity map on $A_j$. In the proof we consider three cases.
	
	In the case 1) we have
	\[ \mathcal{P}_j \epsilon_j(w_1 + w_2h) = \gamma_j \left( (w_1 + w_2b) e_{f_j}(a) \right) = w_1 e_{f_j}(1) + w_2 e_{f_j}(1)h = w_1 + w_2h.  \]
	
	In the case 2), Eq. (\ref{eq:fg}) implies that
	$F_{j, 1}(\alpha_j)G_{1, j}(\alpha_j) = (\alpha_j - 1)^{-1}(\alpha_j + 1)^{-1} = (\alpha_j^2 - 1)^{-1}$. Hence
	\[ 
	    T_1(\alpha_j) = \frac{-M(\alpha_j+\alpha_j^{-1}) + \alpha_j N(\alpha_j + \alpha_j^{-1}) - \alpha_j P(\alpha_j + \alpha_j^{-1}) + \alpha_j^2 Q(\alpha_j+\alpha_j^{-1})}{\alpha_j^2 - 1}, 
	\]
	\[ 
	T_1(\alpha_j^{-1}) = \frac{\alpha_j^2 M(\alpha_j+\alpha_j^{-1}) - \alpha_j N(\alpha_j + \alpha_j^{-1}) + \alpha_j P(\alpha_j + \alpha_j^{-1}) - Q(\alpha_j+\alpha_j^{-1})}{\alpha_j^2 - 1}, 
	\]
	\[
	   T_2(\alpha_j) = \frac{ -M(\alpha_j+\alpha_j^{-1}) + \alpha_j^{-1}N(\alpha_j + \alpha_j^{-1}) - \alpha_j P(\alpha_j+\alpha_j^{-1}) + Q(\alpha_j+\alpha_j^{-1}) }{\alpha_j^2 - 1},
	\]
	\[
	T_2(\alpha_j^{-1}) = \frac{ \alpha_j^2 M(\alpha_j+\alpha_j^{-1}) - \alpha_j^{3}N(\alpha_j + \alpha_j^{-1}) + \alpha_j P(\alpha_j+\alpha_j^{-1}) - \alpha_j^2 Q(\alpha_j+\alpha_j^{-1}) }{\alpha_j^2 - 1},
	\]
	It follows that
	\[ 
	    \mathcal{P}_j \epsilon_j \begin{pmatrix}
	    M(\alpha_j + \alpha_j^{-1}) & P(\alpha_j + \alpha_j^{-1}) \\
	    N(\alpha_j + \alpha_j^{-1}) & Q(\alpha_j + \alpha_j^{-1}) 
	    \end{pmatrix}  =
	    \mathcal{P}_j \left( T_1(a) + bT_2(a)) e_{f_j}(a) \right) =
	    \sigma_j \tau_j  \left(T_1(a) + bT_2(a)) e_{f_j}(a) \right) =
	\]
	\[
	    = Z_j^{-1} \begin{pmatrix}
		T_1(\alpha_j) & T_2(\alpha_j^{-1}) \\
		T_2(\alpha_j) & T_1(\alpha_j^{-1}) 
		\end{pmatrix} Z_j =	
		{ %\color{red}
		\begin{pmatrix} 1 &  -\alpha_j \\ 1 & -\alpha_j^{-1} \end{pmatrix}^{-1}
		}
		\begin{pmatrix}
		T_1(\alpha_j) & T_2(\alpha_j^{-1}) \\
		T_2(\alpha_j) & T_1(\alpha_j^{-1}) 
		\end{pmatrix}
		\begin{pmatrix} 1 &  -\alpha_j \\ 1 & -\alpha_j^{-1} \end{pmatrix}.
	\]
	By direct calculations, we obtain
	\[  \mathcal{P}_j \epsilon_j \begin{pmatrix}
	M(\alpha_j + \alpha_j^{-1}) & P(\alpha_j + \alpha_j^{-1}) \\
	N(\alpha_j + \alpha_j^{-1}) & Q(\alpha_j + \alpha_j^{-1}) 
	\end{pmatrix}
	 = \begin{pmatrix}
	M(\alpha_j + \alpha_j^{-1}) & P(\alpha_j + \alpha_j^{-1}) \\
	N(\alpha_j + \alpha_j^{-1}) & Q(\alpha_j + \alpha_j^{-1}) 
	\end{pmatrix}. \]
	
	In the case 3), similarly, using (\ref{eq:2}) and (\ref{eq:orth}), we get
	\begin{equation} \label{eq:teor3:1}
	\begin{split}
	    \mathcal{P}_j \epsilon_j \begin{pmatrix}
	    M(\alpha_j) & P(\alpha_j^{-1}) \\ N(\alpha_j) & Q(\alpha_j^{-1})
	    \end{pmatrix} =
	    \tau_j \left( [M(a) + bN(a)] e_{f_j}(a) + [Q(a) + bP(a)] e_{f^*_j}(a) \right) = \\
	 =\begin{pmatrix}
		M(\alpha_j) e_{f_j}(\alpha_j) + Q(\alpha_j) e_{f^*_j}(\alpha_j) &
		N(\alpha_j^{-1}) e_{f_j}(\alpha_j^{-1}) + P(\alpha_j^{-1}) e_{f^*_j}(\alpha_j^{-1}) \\
		 N(\alpha_j) e_{f_j}(\alpha_j) + P(\alpha_j)e_{f^*_j}(\alpha_j) &
		 M(\alpha_j^{-1}) e_{f_j}(\alpha_j^{-1}) + Q(\alpha_j^{-1}) e_{f^*_j}(\alpha_j^{-1})
		\end{pmatrix}
	 =	\begin{pmatrix}
		M(\alpha_j) & P(\alpha_j^{-1}) \\ N(\alpha_j) & Q(\alpha_j^{-1})
		\end{pmatrix}. 
	\end{split}
	\end{equation}
	
	Therefore, $P_j \epsilon_j = \mathtt{id}_{A_j}$ for $1 \leq j \leq r+s$. It follows that
	\[ \mathcal{P}_j \left( \sum_{j=1}^{r+s} \epsilon_j \right) = \bigoplus_{j=1}^{r+s} \mathcal{P}_j \epsilon_j = \bigoplus_{j=1}^{r+s} \mathtt{id}_{A_j}. \]
	Hence $\mathcal{P}^{-1} = \sum_{j=1}^{r+s} \epsilon_j$, and $\epsilon_j$ are  $\mathbb{F}_q$--algebras monomorphisms.
\end{proof}

\textbf{Algebraic description of dihedral codes.} In the case $\gcd(q, 2n) = 1$, the dihedral codes were algebraically described in \cite{VedDeu18}, \cite{VedDeu20}. Below, we generalize this result to the case $\gcd(q, n) = 1$. 

\begin{remark} \label{lem:2} 
	Since $\mathbb{F}_q \langle h \mid h^2 \rangle \simeq \mathcal{R}_n = \frac{\mathbb{F}_q[x]}{(x^2-1)}$, it follows that any proper non--zero ideal of $\mathbb{F}_q \langle h \mid h^2 \rangle$ is generated by $1 + h$ or $1 - h$.  
 	Note that if $2 \nmid q$, then these ideals are also generated by idempotents $(1+h)/2$, $(1-h)/2$.  We also note that $\mathbb{F}_q \langle h \mid h^2 \rangle (1 \pm h) = \{ p \pm ph \mid p \in \mathbb{F}_q \}$.
\end{remark}

Consider (\ref{eq:2}). Let
\[
	F_j := \begin{cases}
		\mathbb{F}_q, & 1 \leq j \leq \zeta(n),  \\
		\mathbb{F}_q[\alpha_j+\alpha_j^{-1}],& \zeta(n)+1 \leq j \leq r,\\
		\mathbb{F}_q[\alpha_j], & r+1 \leq j \leq r+s.
	\end{cases}
\]
\begin{equation} 
	\label{eq:a1}
	I_j(x, y) := \begin{cases}
		\mathbb{F}_q \langle h \mid h^2 \rangle \left[ (x+y) + (x-y)h \right], & 1 \leq j \leq \zeta(n), \\
		\left\{
		\begin{pmatrix}
			ky & -kx \\ ty & -tx
		\end{pmatrix} \:| \; k,t \in F_j 
		\right\}, & \zeta(n)+1 \leq j \leq r+s.
	\end{cases}
\end{equation}
\begin{remark} \label{remark:4}
Note that  $I_j(x, y) = I_j(\mu x, \mu y)$ for any non--zero $\mu \in F_j$. In addition, for $\zeta(n)+1 \leq j \leq r+s$ we have
\begin{equation} \label{eq:Ij}
	I_j(x, y) = A_j \begin{pmatrix}
		y & -x \\ 0 & 0
	\end{pmatrix} = A_j \begin{pmatrix}
		0 & 0 \\ y & -x
	\end{pmatrix}.
\end{equation}
\end{remark}
\begin{theorem} \label{theor:codes}
	Let $\mathrm{gcd}(q, n) = 1$. Consider the decomposition (\ref{eq:2}) of $\mathbb{F}_q D_{n}$. Then for any code $I \subset \mathbb{F}_qD_{n}$ we have
	\begin{equation} \label{eq:form}
		\mathcal{P}(I) = \bigoplus_{j=1}^{r+s} B_j,
		\quad
		B_j = \begin{cases}
			A_j, & j \in J_1 \\
			I_j(0, 1), & j \in J_2 \\
			I_j(1, 0), & j \in J_3 \\
			I_j(-x_j, 1) & j \in J_4 \\			
			0, & j \notin J_1 \cup J_2	\cup J_3 \cup J_4
		\end{cases}
	\end{equation}
	where  $x_j \in F_j$, $x_j \neq 0$, and $J_1, J_2, J_3, J_4$ are pairwise disjoint sets such that
	\begin{equation} \label{eq:15}
		J_1, J_2, J_3 \subset \{ 1, \dots, r+s \}, \quad J_4 \subset \{ \zeta(n) + 1, \dots, r+s \}
	\end{equation}
	and $1 \notin J_2$ if $q$ is even. 
	
	On the other hand, for any $x_j, y_j \in F_j$ and for any $B_j$ such that $B_j = A_j$ or $B_j=I_j(x_j, y_j)$ or $B_j = 0$, the set
	$\mathcal{P}^{-1}(\bigoplus_{j=1}^{r+s} B_j)$
	is a $D_n$--code.
\end{theorem}
\begin{proof}
	Since $\mathbb{F}_qD_n$ and $\bigoplus_{j=1}^{r+s} A_j$ are isomorphic, it follows that there is one--to--one correspondence between the $D_n$--codes and the left ideals of $\bigoplus_{j=1}^{r+s} A_j$. As is well--known, any left ideal of a direct sum of algebras with identity is a direct sum of ideals of the direct summands. So, to prove the theorem, it is sufficient to show that any ideal of $A_j$ is of the form $B_j$.
		
	Indeed, for $1 \leq j \leq \zeta(n)$, Remark \ref{lem:2} implies that any ideal of $A_j$ is of the form $A_j, 0, I_j(0,1)$ or $I_j(1,0)$; and $I_j(0,1)=I_j(1,0)$ if $2 \mid q$. For $\zeta(n)+1 \leq j \leq r+s$, Theorem 3 of \cite{VedDeu18} and Remark \ref{remark:4} imply that any left ideal of $A_j$ is of the form $0, A_j, I_j(0,1), I_j(1,0)$, $I_j(-x_j,1)$, $x_j \in F_j$.
\end{proof}

\begin{corollary}
	Let $\gcd(q,n) = 1$. Lemma \ref{lem:2}, (\ref{eq:Ij}) and Theorem \ref{theor:codes} imply that any left ideal of $\mathbb{F}_qD_{n}$ is also of the form $(\mathbb{F}_qD_{n})i$, $i \in \mathbb{F}_qD_{n}$.
\end{corollary}

\section{Dual dihedral codes} \label{sec:dual}
In the classical coding theory for every linear code $C \subset \mathbb{F}_q^n$ the dual code is given by
\[ 
C^{\perp} = \{ x \in \mathbb{F}_q^n \mid \forall c \in C \quad \mathtt{dot}_{\mathbb{F}_q^n}(x, c) = 0 \},
\]
where
\[
\mathtt{dot}_{\mathbb{F}_q^n}: \mathbb{F}_q^n \times \mathbb{F}_q^n \to \mathbb{F}_q, \quad  (x, c) \mapsto \sum_{i=1}^{n} x_i c_i 
\]
is the dot product in $\mathbb{F}_q^n$. In the case when $C = C^{\perp}$ the code $C$ is called \textit{self--dual}. 

Similarly, we can define the dot product in $\mathbb{F}_qG$:
\[
\mathtt{dot}_{\mathbb{F}_qG}: \mathbb{F}_qG \times \mathbb{F}_qG \to \mathbb{F}_q, \quad  \left( \sum_{g \in G} x_g g, \sum_{g \in G} y_g g \right) \mapsto \sum_{g \in G} x_g y_g. 
\]
Let $E_{|G|} = \{ e_1, \dots, e_n \}$ be the standard basis of $\mathbb{F}_q^{|G|}$ and let $\phi: G \to E_{|G|}$ be an arbitrary bijection. The map $\phi$ can be extended to the linear isomorphism $\mathbb{F}_qG \to \mathbb{F}_q^{|G|}$. Clearly, we have
$\mathtt{dot}_{\mathbb{F}_q^{|G|}} (\phi(x), \phi(y)) = \mathtt{dot}_{\mathbb{F}_qG}(x, y)$.

Note that there is a connection between the multiplication in $\mathbb{F}_qG$ and the dot product $\mathtt{dot}_{\mathbb{F}_qG}$. Indeed, let $x, y \in \mathbb{F}_qG$; then $x y^* = \sum_{g \in G} \left( \mathtt{dot}_{\mathbb{F}_qG}(g^{-1} x, y) \right) g$.

Let $C \subset \mathbb{F}_qG$ be a group code; consider its right annihilator
$
\mathtt{rAnn}_{\mathbb{F}_qG}(C) = \{ x \in \mathbb{F}_qG \mid \forall c \in C \quad cx = 0 \}.
$
\begin{theorem}[\cite{CouGonMar12}, Theorem 1.6] \label{theor:1}
	Let $C \subset \mathbb{F}_qG$ be a group code and let $C^{\perp} = \{ x \in \mathbb{F}_qG \mid \forall c \in C \quad \mathtt{dot}_{\mathbb{F}_qG}(c, x) = 0 \}$ be its dual code. Then
	\( C^{\perp} = \left( \mathtt{rAnn}_{\mathbb{F}_qG}(C) \right)^*.  \)
\end{theorem}

In this section, for any dihedral code, the image of its dual code under (\ref{eq:2}) is described. Theorem \ref{theor:1} provides a general framework to do this. So, to find the image of a dual code we have to describe the connection between $\mathcal{P}$ and the antiautomorphism $\mathrm{inv}_{\mathbb{F}_qD_{n}}: x \mapsto x^*$ (see (\ref{eq:anti})). We also need to describe the right annihilators of dihedral codes.

For $\zeta(n) + 1 \leq j \leq r+s$ let us define the maps $\mathcal{H}_j: \mathbb{M}_2(\mathbb{F}_q[\alpha_j]) \to \mathbb{M}_2(\mathbb{F}_q[\alpha_j])$ as follows
\begin{equation}
\mathcal{H}_j: \quad
\begin{pmatrix}
M & P \\ N & Q
\end{pmatrix} \mapsto
\begin{pmatrix}
Q & P \\ N & M
\end{pmatrix}.
\end{equation}
Let us consider (\ref{eq:zj}) and the decomposition (\ref{eq:2}) and let
\begin{equation} \label{eq:nu}
\nu := \bigoplus_{j=1}^{r+s} \nu_j, \quad
\nu_j := \begin{cases}
\mathrm{id}_{A_j}, & 1 \leq j \leq \zeta(n) \\
\sigma_j \mathcal{H}_j \sigma_j^{-1}, & \zeta(n) + 1 \leq j \leq r \\
\mathcal{H}_j, & r+1 \leq j \leq r+s
\end{cases}.
\end{equation}
\begin{Lemma} \label{lem:3}
	Let $\gcd(q, n) = 1$. Then the following diagram is commutative.
	\begin{equation} \label{eq:diag}
     \begin{array}{ccc}
		\mathbb{F}_qD_{n} & \xrightarrow{\mathrm{inv}_{\mathbb{F}_qD_{n}}} & \mathbb{F}_qD_{n} \\
		\Bigg \downarrow \mathcal{P} &  & \Bigg \downarrow \mathcal{P}  \\
		\bigoplus_{j=1}^{r+s} A_j & \xrightarrow{\quad \nu \quad} & \bigoplus_{j=1}^{r+s} A_j.
	\end{array}
	\end{equation}
\end{Lemma}
\begin{proof}
	Let $u = P(a) + bQ(a)$ be an arbitrary element of $\mathbb{F}_qD_{n}$. Let us consider $\gamma_j$, $\tau_j$ from (\ref{eq:2}) and $\mathrm{inv}_{\mathbb{F}_qD_{n}}$ from (\ref{eq:anti}).  Using (\ref{eq:dih}), we obtain $(ba^{i})^{-1} = (ba^{i})$ and
	\[u^* = P(a^{-1}) + bQ(a). \]
	Let us consider several cases.
	
	Case 1: $1 \leq j \leq \zeta(n)$. Since $\alpha_j = \pm 1$, it follows that  
	\[
	\gamma_j(u) = P(\alpha_j) + Q(\alpha_j) h, \quad \gamma_j(u^*) = P(\alpha_j^{-1}) + Q(\alpha_j) h = P(\alpha_j) + Q(\alpha_j) h = \gamma_j(u).
	\]
	Hence
	$\mathcal{P}_j \mathrm{inv}_{\mathbb{F}_qD_{n}} = \gamma_j \mathrm{inv}_{\mathbb{F}_qD_{n}} = \gamma_j = \nu_j \gamma_j = \nu_j \mathcal{P}_j$
	(see (\ref{eq:2}), (\ref{eq:nu})).
	
	Case 2: $\zeta(n) + 1 \leq j \leq r+s$. We have
	\begin{equation} \label{eq:23}
	\tau_j(u) =  
	\begin{pmatrix} P(\alpha_j) & Q(\alpha_j^{-1}) \\ Q(\alpha_j) & P(\alpha_j^{-1}) \end{pmatrix}, \quad
	\tau_j(u^*) =  
	\begin{pmatrix} P(\alpha_j^{-1}) & Q(\alpha_j^{-1}) \\ Q(\alpha_j) & P(\alpha_j) \end{pmatrix} = \mathcal{H}_j(\tau_j(u)).
	\end{equation}
	Using (\ref{eq:2}), (\ref{eq:nu}), (\ref{eq:23}), we obtain
	\[ \mathcal{P}_j(u^*) = \sigma_j \tau_j (u^*) = \sigma_j \mathcal{H}_j \tau_j(u), \quad
	 \nu_j \mathcal{P}_j(u) = \left( \sigma_j \mathcal H_j \sigma_j^{-1} \right) \left( \sigma_j \tau_j \right) (u) = \sigma_j \mathcal{H}_j \tau_j(u)  \]
	for $\zeta(n) + 1 \leq j \leq r$, and
	\[ \mathcal{P}_j(u^*) = \tau_j(u^*) = \mathcal{H}_j \tau_j(u) = \mathcal{H}_j \mathcal{P}_j(u)  \]
	for $r+1 \leq j \leq r+s$. Hence $\mathcal{P}_j \mathrm{inv}_{\mathbb{F}_qD_{n}} = \nu_j \mathcal{P}_j$.	
	
	So, since $\mathcal{P} = \bigoplus_{j=1}^{r+s}\mathcal{P}_j$, $\nu = \bigoplus_{j=1}^{r+s} \nu_j$, and $P_j \mathrm{inv}_{\mathbb{F}_qD_{n}} = \nu_j P_j$, $1 \leq j \leq r+s$, it follows that diagram (\ref{eq:diag}) is commutative.
\end{proof}

In the following Lemmas, the right annihilators of the left ideals of $A_j$ and $\bigoplus_{j=1}^{r+s}A_j$ are described.

\begin{Lemma} \label{lem:4}
	Consider the algebra $\mathbb{F}_q \langle h \mid h^2 \rangle$. Then
	
	(i) $\mathtt{rAnn}_{\mathbb{F}_q \langle h \mid h^2 \rangle}\left( \mathbb{F}_q \langle h \mid h^2 \rangle (1+h) \right) = (1-h) \mathbb{F}_q \langle h \mid h^2 \rangle = \mathbb{F}_q \langle h \mid h^2 \rangle (1-h)$;
	
	(ii) $\mathtt{rAnn}_{\mathbb{F}_q \langle h \mid h^2 \rangle}\left( \mathbb{F}_q \langle h \mid h^2 \rangle (1-h) \right) = (1+h) \mathbb{F}_q \langle h \mid h^2 \rangle = \mathbb{F}_q \langle h \mid h^2 \rangle (1+h)$;
	
	(iii) $\mathtt{rAnn}_{\mathbb{F}_q \langle h \mid h^2 \rangle}\left( \mathbb{F}_q \langle h \mid h^2 \rangle \right) = 0$;
	
	(iv) $\mathtt{rAnn}_{\mathbb{F}_q \langle h \mid h^2 \rangle}\left( 0 \right) = \mathbb{F}_q \langle h \mid h^2 \rangle$.
\end{Lemma}
\begin{proof}
	As is well--known, a right annihilator is a right ideal of a ring. Since $\mathbb{F}_q \langle h \mid h^2 \rangle$  is commutative, it follows that its one--sided ideals are two--sided. Lemma \ref{lem:2}  implies that the ideals of $\mathbb{F}_q \langle h \mid h^2 \rangle$ are of the form
	\[ 0, \quad \mathbb{F}_q \langle h \mid h^2 \rangle, \quad \mathbb{F}_q \langle h \mid h^2 \rangle (1+h), \quad \mathbb{F}_q \langle h \mid h^2 \rangle (1-h).  \]
	Therefore, the relation $(1+h)(1-h) = (1 - 1) + (h - h) = 0$ completes the proof.
\end{proof}
\begin{remark}
	If $\mathtt{char}(\mathbb{F}_q) = 2$, then $1-h = 1+h$. Hence the ideal $\mathbb{F}_q \langle h \mid h^2 \rangle(1+h)$ annihilates itself.
\end{remark}

It is easy to see that following two Lemmas are also true. We give them without proof.
\begin{Lemma} \label{lem:5}
	Let $\mathbb{F}$ be a field. Consider the algebra $\mathbb{M}_2(\mathbb{F})$. Let $x, y \in F$, then
	
	(i) $\mathtt{rAnn}_{\mathbb{M}_2(\mathbb{F})} \left(
	\left\{ 
	\begin{pmatrix}
	ky & -kx \\ ty & -tx
	\end{pmatrix} \mid k,t \in F
	\right\}
	\right) =
	\left\{ 
	\begin{pmatrix}
	kx & tx \\ ky & ty
	\end{pmatrix} \mid k,t \in F
	\right\};	
	$
	
	(ii) $\mathtt{rAnn}_{\mathbb{M}_2(\mathbb{F})} (\mathbb{M}_2(\mathbb{F})) = 0$;
	
	(iii) $\mathtt{rAnn}_{\mathbb{M}_2(\mathbb{F})} (0) = \mathbb{M}_2(\mathbb{F})$.
\end{Lemma}
\begin{Lemma} \label{lem:6}
	Let $R_1, \dots, R_t$ be associative rings with identity, and let
	$$R = R_1 \oplus R_2 \oplus \dots \oplus R_t.$$
	Let $S$ be a left ideal of $R$ and $S = S_1 \oplus \dots S_t$ be its decomposition into left ideals of the direct summands $R_1, \dots, R_t$. Then
	\[ \mathtt{rAnn}_{R}(S) = \mathtt{rAnn}_{R_1}(S_1) \oplus \mathtt{rAnn}_{R_2}(S_2) \oplus \dots \oplus \mathtt{rAnn}_{R_t}(S_t),   \]
	where $\mathtt{rAnn}_{R_i}(S_i) \subset R_i$ are the right annihilators of $S_i \subset R_i$.
\end{Lemma}

The following theorem describes the dual codes in terms of decomposition (\ref{eq:2}).
\begin{theorem} \label{theor:dual}
	Let $\mathrm{gcd}(q, n) = 1$, let $I \subset \mathbb{F}_qD_{n}$ be a dihedral code. Consider the decomposition (\ref{eq:2}) of $\mathbb{F}_q D_{n}$  and the image (\ref{eq:form})-- (\ref{eq:15}) of $I$ under this decomposition: 
	\[
	\mathcal{P}(I) = \bigoplus_{j=1}^{r+s} B_j, \quad	
	B_j = \begin{cases}
	A_j, & j \in J_1 \\
	I_j(0, 1), & j \in J_2 \\
	I_j(1, 0), & j \in J_3 \\
	I_j(-x_j, 1) & j \in J_4 \\			
	0, & j \notin J_1 \cup J_2	\cup J_3 \cup J_4
	\end{cases}. \] 
	Let $I^{\perp} \subset \mathbb{F}_qD_{n}$ be the dual code of $I$. Then
	\[ \mathcal P(I^{\perp}) = \bigoplus_{j=1}^{r+s} \hat B_j, \]
	where
	\begin{enumerate}
		\item[1)] for $1 \leq j \leq \zeta(n)$:
		\[ \hat{B}_j = \begin{cases}
		A_j, & \text{if } B_j = 0 \\
		I_j(0, 1), & \text{if } B_j = I_j(1, 0) \\
		I_j(1, 0), & \text{if } B_j = I_j(0, 1) \\
		0, & \text{if } B_j = A_j
		\end{cases} \]	
		
		\item[2)] for $\zeta(n) + 1 \leq j \leq r$:
		\[
		\hat{B}_j = 
		\begin{cases}
		A_j, & \text{if } B_j = 0 \\		
		I_j((\alpha_j+\alpha_j^{-1}), 2), & \text{if } B_j = I_j(1, 0) \\
		I_j\left( 2, (\alpha_j+\alpha_j^{-1}) \right), & \text{if } B_j = I_j(0, 1) \\
		I_j\left( 2 + (\alpha_j + \alpha_j^{-1})x_j, (\alpha_j+\alpha_j^{-1}) + 2x_j \right), & \text{if } B_j = I_j(-x_j, 1) \\
		0, & \text{if } B_j = A_j
		\end{cases}
		\]
		
		\item[3)] for $r+1 \leq j \leq r+s$:
		\[
		\hat{B}_j = \begin{cases}
		A_j, & \text{if } B_j = 0 \\
		I_j(0, 1), & \text{if } B_j = I_j(0, 1)  \\
		I_j(1, 0), & \text{if } B_j = I_j(1, 0)  \\
		I_j(x_j, 1) & \text{if } B_j = I_j(-x_j, 1) \\
		0, & \text{if } B_j = A_j
		\end{cases}
		\]
	\end{enumerate}
\end{theorem}
\begin{proof}
	Let us consider the algebra $\Delta := \bigoplus_{j=1}^{r+s} A_j$. Using Theorem \ref{theor:1} and Lemma \ref{lem:3}, we obtain	
	\[ 	\mathcal{P}(I^{\perp}) = 
	\mathcal{P} \left( \left( \mathtt{rAnn}_{\mathbb{F}_qD_{n}}(I) \right)^{*} \right) =
	\nu \mathcal{P} \left( \mathtt{rAnn}_{\mathbb{F}_qD_{n}}(I)  \right) =
	\nu \left( \mathtt{rAnn}_{\Delta} (\mathcal{P}(I)) \right).
	\]	
	Using Lemma \ref{lem:6}, we get
	\[ 	\mathtt{rAnn}_{\Delta} (\mathcal{P}(I)) = 
		\mathtt{rAnn}_{\Delta} \left( \bigoplus_{j=1}^{r+s} B_j \right) =
		\bigoplus_{j=1}^{r+s} \mathtt{rAnn}_{A_j} (B_j),
	\]
	So, using the definition of $\nu$,  we conclude that
	\[ \mathcal{P}(I^{\perp}) = 
		\nu \left( \bigoplus_{j=1}^{r+s} \mathtt{rAnn}_{A_j} (B_j) \right) =
		\bigoplus_{j=1}^{r+s} \nu_j \left( \mathtt{rAnn}_{A_j} (B_j) \right).
	\]
	Thus, to compute $\mathcal{P}(I^{\perp})$ we need to find $\nu_j \left( \mathtt{rAnn}_{A_j} (B_j) \right)$, $1 \leq j \leq r+s$. Let us consider several cases.
		
	Case 1): $1 \leq j \leq \zeta(n)$. Lemma \ref{lem:4} and (\ref{eq:a1}) imply
	\[ 
		\hat B_j = \nu_j \left( \mathtt{rAnn}_{A_j} (B_j) \right) = 
		\mathrm{id}_{A_j} \left( \mathtt{rAnn}_{A_j} (B_j) \right) =
		\mathtt{rAnn}_{A_j} (B_j) =
		\begin{cases}
		A_j, & \text{if } B_j = 0 \\
		I_j(0, 1), & \text{if } B_j = I_j(1, 0) \\
		I_j(1, 0), & \text{if } B_j = I_j(0, 1) \\
		0, & \text{if } B_j = A_j
		\end{cases}.
	\]
	Claim 1) of the theorem is proved.
	
	Case 2): $\zeta(n) + 1 \leq j \leq r$. Recall that $A_j = \mathbb{M}_2(\mathbb{F}_q[\alpha_j + \alpha_j^{-1}])$ and $\nu_j = \sigma_j \mathcal{H}_j \sigma_j^{-1}$ in this case. Below, we use Lemma \ref{lem:5}, which describes the right annihilator of the left ideals of $A_j$. If $B_j = A_j$ we clearly have
	\[ \hat B_j = \nu_j \left( \mathtt{rAnn}_{A_j} (A_j) \right) = \nu_j (0) = 0. \]
	If $B_j = 0$, then 
	\[
	\hat B_j = \nu_j \left( \mathtt{rAnn}_{A_j} (0) \right) = \nu_j (A_j) = A_j.  
	\]	
	Now let $B_j = I_j(-x_j, 1)$. Obviously, (\ref{eq:Ij}) implies
	\[ I_j(-x_j, 1) = 
		A_j \begin{pmatrix}
		1 & x_j \\ 0 & 0
		\end{pmatrix}, 
		\quad
		\mathtt{rAnn}_{A_j} \left( I_j(-x_j, 1) \right) =
		\begin{pmatrix}
		0 & -x_j \\ 0 & 1
		\end{pmatrix} A_j.
	\]
	It follows that
	\[ 
		\hat B_j = \nu_j \left( \mathtt{rAnn}_{A_j} \left( I_j(-x_j, 1) \right)  \right) =
		\sigma_j \mathcal{H}_j \sigma_j^{-1} \left( 
		\begin{pmatrix}
		0 & -x_j \\ 0 & 1
		\end{pmatrix} A_j 
		\right).
	\]
	By direct calculation we obtain
	\[ 
		\hat B_j = \nu_j \left( \mathtt{rAnn}_{A_j} \left( I_j(-x_j, 1) \right)  \right) =   
		A_j \begin{pmatrix}
		-2 x_j - (\alpha_j + \alpha_j^{-1}) & (\alpha_j + \alpha_j^{-1}) x_j + 2 \\ 0 & 0
		\end{pmatrix}.
	\]
	So, using (\ref{eq:Ij}), we see that
	\[ \hat B_j = \nu_j \left( \mathtt{rAnn}_{A_j} \left( I_j(-x_j, 1) \right) \right) = I_j \left( 2 + (\alpha_j + \alpha_j^{-1})x_j, (\alpha_j+\alpha_j^{-1}) + 2x_j \right).
	\]
	In particular, if $B_j = I_j(0, 1)$, then
	$ \hat B_j = \nu_j \left( \mathtt{rAnn}_{A_j} \left( I_j(0, 1) \right) \right) = I_j \left( 2, (\alpha_j+\alpha_j^{-1}) \right)$.	Similarly, if $B_j = I_j(1, 0)$, we have
	\[ I_j(1, 0) = 
	A_j \begin{pmatrix}
	0 & 0 \\ 0 & 1
	\end{pmatrix}, 
	\quad
	\mathtt{rAnn}_{A_j} \left( I_j(1, 0) \right) =
	\begin{pmatrix}
	1 & 0 \\ 0 & 0
	\end{pmatrix} A_j.
	\]
	Hence
	\[ 
	\hat B_j = \nu_j \left( \mathtt{rAnn}_{A_j} \left( I_j(1, 0) \right)  \right) =
	\sigma_j \mathcal{H}_j \sigma_j^{-1} \left( 
	\begin{pmatrix}
	1 & 0 \\ 0 & 0
	\end{pmatrix} A_j 
	\right) =
	A_j \begin{pmatrix}
	-2 & (\alpha_j+\alpha_j^{-1}) \\ 0 & 0
	\end{pmatrix} 	
	 = I_j((\alpha_j+\alpha_j^{-1}), 2).
	\]
	Claim 2) of the theorem is proved.
	
	Case 3): $r+1 \leq j \leq r+s$. Using Lemma \ref{lem:5}, (\ref{eq:nu}) and (\ref{eq:a1}), for any $x,y \in F_j$ we obtain
	\[ \nu_j \left( \mathtt{rAnn}_{A_j} \left( I_j(x, y) \right) \right) = 
	\mathcal{H}_j \left(
	\left\{
	\begin{pmatrix}
	kx & tx \\ ky & ty 
	\end{pmatrix} \mid k, t \in F_j
	\right\} \right) = 
	\left\{
	\begin{pmatrix}
	ty & tx \\ ky & kx 
	\end{pmatrix} \mid k, t \in F_j
	\right\} = I_j(-x, y).
	\]
	Futher, Lemma \ref{lem:5} implies
	\[  
	 \nu_j \left( \mathtt{rAnn}_{A_j} (A_j) \right) = \nu_j (0) = 0, \quad
	 \nu_j \left( \mathtt{rAnn}_{A_j} (0) \right) = \nu_j (A_j) = A_j.  
	\]
	Claim 3) of the theorem is proved.
\end{proof}
\begin{corollary} \label{col:2}
	Let $2 \mid q$. Then a code $I \subset \mathbb{F}_qD_{n}$ is self--dual if and only if $B_j \neq 0$, $B_j \neq A_j$ for all $1 \leq j \leq r+s$.
\end{corollary}
\begin{corollary}
	Let $2 \nmid q$. Then $I \neq I^{\perp}$ for any code $I \subset \mathbb{F}_qD_{n}$.
\end{corollary}

\section{Bases and generating matrices of dihedral codes} \label{sec:matr}
The codes in $\mathbb{F}_qD_{n}$ were described in Theorem \ref{theor:codes}. In this section, $\mathbb{F}_q$--bases, generating and check matrices of these codes are constructed. Below, we keep using the definitions and the notation from the previous sections.

Let $E_{2n} = \{ \textbf{e}_1, \dots, \textbf{e}_{2n} \}$ be a standard basis of $\mathbb{F}_q^{2n}$. Let $\phi: D_{n} \to E_{2n}$ be a bijective map, given by
\begin{equation} \label{eq:phi}
\phi(a^i) = \textbf{e}_{i+1}, \quad \phi(b a^i) = \textbf{e}_{n + i + 1} \quad (i = 0, \dots, n-1). 
\end{equation}
The map $\phi$ can naturally be extended to the linear isomorphism $\phi: \mathbb{F}_q D_{n} \to \mathbb{F}_q^{2n}$. Let $I \subset \mathbb{F}_qD_{n}$ be a dihedral code, and let $S = \{ s_1, \dots, s_k \}$ be a basis of it. The matrix
\[ \textbf{G}_I = \begin{pmatrix}
\rule[.5ex]{2.5ex}{0.5pt} & \phi(s_1) & \rule[.5ex]{2.5ex}{0.5pt} \\
\rule[.5ex]{2.5ex}{0.5pt} & \phi(s_2) & \rule[.5ex]{2.5ex}{0.5pt} \\
& \vdots & \\
\rule[.5ex]{2.5ex}{0.5pt} & \phi(s_k) & \rule[.5ex]{2.5ex}{0.5pt} \\
\end{pmatrix} \]
is called a generating matrix of $I$. 

By $\langle a \rangle_{D_{n}}$ we denote the cyclic subgroup of order $n$ generated by $a \in D_n$. It is also convinient to consider linear isomorphism $\psi: \mathbb{F}_q \langle a \rangle_{D_{n}} \to \mathbb{F}_q^n$, given by
\begin{equation} \label{eq:psi}
 \psi(a^i) = \textbf{e}_{i + 1} \quad (
\textbf{e}_i \in E_n, \; i = 0, \dots, n-1).
\end{equation}
Clearly, for an element $u = P(a) + bQ(a) \in \mathbb{F}_qD_{n}$ we have
\begin{equation} \label{eq:phipsi}
\phi(u) = \begin{pmatrix}
\psi(P(a)), & \psi(Q(a))
\end{pmatrix}, \quad
\phi(au) = \begin{pmatrix}
\psi(aP(a)), & \psi(a^{-1} Q(a))
\end{pmatrix}, \quad
\phi(bu) = \begin{pmatrix}
\psi(Q(a)), & \psi(P(a))
\end{pmatrix}.
\end{equation}
Let $\mathtt{rshift}$ be the operator of right cyclic shift on $\mathbb{F}_q^n$. Note that
\[  \phi(au) = \begin{pmatrix}
\mathtt{rshift} \left( \psi(P(a)) \right), & \mathtt{rshift}^{-1} \left( \psi(Q(a)) \right) \end{pmatrix}.  \]

Consider the isomorphism $\mathcal{P}$ from Theorem \ref{theor:decomp} and the maps $\epsilon_j$ from Theorem \ref{theor:3}. Let $S$ be a subset of $A_j$; by $\overline{ S_j}$ we denote
$\overline{ S_j}  = \epsilon_j(S)$.
Recall that $\mathcal P_j \epsilon_j = \mathrm{id}_{A_j}$ and $\mathcal P_j \epsilon_i = 0$ for $i \neq j$ (see (\ref{eq:eps})). Note that
\[ \overline{S_j} = \mathcal P^{-1} (0 \oplus \dots \oplus 0 \oplus S \oplus 0 \oplus \dots \oplus 0). \]

\begin{Lemma} \label{lem:7}
	Let $\gcd(q, n) = 1$ and let $1 \leq j \leq \zeta(n)$. Then
	\begin{enumerate}
		\item[1)] $\mathcal{B}(A_j) := \{ e_{f_j}(a),  be_{f_j}(a) \}$ is a $\mathbb{F}_q$--basis of $\overline{A}_j$;
		\item[2)] $\mathcal{B}(I_j(1, 0) ) := \{ e_{f_j}(a) + b e_{f_j}(a) \}$ is a $\mathbb{F}_q$--basis of $\overline{ I_j(1, 0) }$;
		\item[3)] $\mathcal{B}(I_j(0, 1)) := \{ e_{f_j}(a) - b e_{f_j}(a) \}$ is a $\mathbb{F}_q$--basis of $\overline{ I_j(0, 1) }$.
	\end{enumerate}
\end{Lemma}
\begin{proof}
	If $1 \leq j \leq \zeta(n)$, then $A_j$ is of the form $\mathbb{F}_q \langle h \mid h^2 \rangle$. Since the set $\left\{ 1, h \right\}$ is a basis of $A_j$, using Theorem \ref{theor:3}, we obtain
	\[ 		 
		 \epsilon_j( \{1, h\} )  =
		\{ e_{f_j}(a),  be_{f_j}(a) \} = 
		\mathcal B( A_j ) .
	\]
	Hence $\mathcal B( A_j )$ is a basis of $\overline{ A_j }$.  Claim 1) is proved.
	
	By definition of $I_j(x, y)$ we have
	\[
	I_j(1,0) = (1 + h) \mathbb{F}_q \langle h \mid h^2 \rangle = 
	\{ P + P h \mid P \in \mathbb{F}_q \} =
	\{ P (1+h) \mid P \in \mathbb{F}_q \},
	\]
	\[
	I_j(0,1) = (1 - h) \mathbb{F}_q \langle h \mid h^2 \rangle = 
	\{ P - P h \mid P \in \mathbb{F}_q \} =
	\{ P (1-h) \mid P \in \mathbb{F}_q \}.
	\]
	It follows that $\{ 1+h \}$ is a basis of $I_j(1, 0)$, and $\{ 1-h \}$ is a basis of $I_j(0, 1)$. Using Theorem \ref{theor:3}, we obtain
	\[ 
		\epsilon_j(\{1 + h\}) =
		\{ e_{f_j}(a) + be_{f_j}(a) \} = 
		\mathcal B( I_j(1, 0) ), 
	\]
	\[ 
		\epsilon_j(\{1 - h\}) =
		\{ e_{f_j}(a) - be_{f_j}(a) \} = 
		\mathcal B( I_j(0, 1) ),
	\]
	So, claims 2) and 3) are proved.
\end{proof}

	Note that the sets $\overline{A_j}$, $\overline{I_j(1, 0)}$, $\overline{I_j(0, 1)}$ are dihedral codes as well. Using formulas (\ref{eq:phi}), (\ref{eq:psi}), (\ref{eq:phipsi})), we obtain the following corollary.
\begin{corollary}	\label{col:3}
Let $1 \leq j \leq \zeta(n)$. Then
	\[ \textbf{G}_{\overline{A_j}} 
		= \begin{pmatrix}
		\phi(e_{f_j}(a))  \\  \phi(be_{f_j}(a))
		\end{pmatrix} =
		\begin{pmatrix}
		\psi(e_{f_j}(a)) & 0 \\ 
		0 & \psi(e_{f_j}(a))
		\end{pmatrix};
	\]
	\[ 
		\textbf{G}_{\overline{I_j(1, 0)}} = 
		\begin{pmatrix}
		\psi(e_{f_j}(a)), & \psi(e_{f_j}(a))
		\end{pmatrix} = 
		\begin{pmatrix}
		1, & 1
		\end{pmatrix} \textbf{G}_{\overline{A_j}};
	\]
	\[ 
		\textbf{G}_{\overline{I_j(0, 1)}} = 
		\begin{pmatrix}
		\psi(e_{f_j}(a)), & -\psi(e_{f_j}(a))
		\end{pmatrix} = 
		\begin{pmatrix}
		1, & -1
		\end{pmatrix} \textbf{G}_{\overline{A_j}}.
	\]
\end{corollary}

Let $\zeta(n) + 1 \leq j \leq r$. Let us recall the Bezout relations  (\ref{eq:fg}) for $\gcd(f_j(x), x-1)=1$ and $\gcd(f_j(x), x+1)=1$:
\[
(x-1) F_{j,1}(x) + f_j(x)F_{j,2}(x) = 1, \quad
(x+1) G_{j,1}(x) + f_j(x)G_{j, 2}(x) = 1.
\]
\begin{Lemma} \label{lem:10}
	Let $\gcd(q, n) = 1$ and $\zeta(n) + 1 \leq j \leq r$. Then
	\begin{enumerate}
		\item[1)] the set
		\[ 
		\mathcal{B}( A_j ) := \left\{ a^i b^k e_{f_j}(a) \mid i = 0 \dots \deg(f_j) - 1, \; k = 0,1 \right\}  
		\] 
		is a $\mathbb{F}_q$--basis of $\overline{ A_j }$;
		\item[2)] for any $u(x), v(x) \in \mathbb{F}_q[x]$ such that $(v(\alpha_j+\alpha_j^{-1}), u(\alpha_j+\alpha_j^{-1})) \neq (0, 0)$, the set
		\begin{equation*}
		\begin{split}
		\mathcal{B}\left( I_j\left(u(\alpha_j + \alpha_j^{-1}), v(\alpha_j + \alpha_j^{-1})\right) \right) := \\ :=
		\left\{ 
		a^i (1+b) F_{j,1}(a)G_{j,1}(a) \left( au(a+a^{-1}) - v(a+a^{-1}) \right) e_{f_j}(a) \mid 
		i = 0 \dots \deg(f_j) - 1
		\right\}
		\end{split}
		\end{equation*}
		is a $\mathbb{F}_q$--basis of $\overline{I_j\left(u(\alpha_j + \alpha_j^{-1}), v(\alpha_j + \alpha_j^{-1})\right)}$.
	\end{enumerate}
\end{Lemma}
\begin{proof}
	Recall that $A_j = \mathbb{M}_2(\mathbb{F}_q[\alpha_j + \alpha_j^{-1}])$ and $\dim_{\mathbb{F}_q}(A_j) = 2 \deg(f_j)$ for $\zeta(r)+1 \leq j \leq r$ (see (\ref{eq:dim}), (\ref{eq:2})). Let us prove claim 1). Formulas (\ref{eq:2}), (\ref{eq:orth}) imply that
	\[
			 \mathcal{P}_j \left( \mathcal{B}(A_j) \right) = 
			 \sigma_j \tau_j \left( \mathcal{B}(A_j) \right) =
			 \sigma_j\left(
			 \left\{ 
			 \begin{pmatrix}
			 \alpha_j^i & 0 \\ 0 & \alpha_j^{-i}
			 \end{pmatrix},
			 \begin{pmatrix}
			 0 & \alpha_j^{-i} \\ \alpha_j^i & 0
			 \end{pmatrix}
			 \mathlarger \mid \; i = 1, \dots, \deg(f_j) - 1
			 \right\}
			 \right),
	\]
	and $\mathcal{P}_i \left( \mathcal{B}(A_j) \right) = \{ 0 \}$ if $i \neq j$. Since $f_j = f_j^*$ is the minimal polynomial of $\alpha_j$ and $\alpha_j^{-1}$ (see section 2) and since $\sigma_j$ is an automorphism of $M_2(\mathbb{F}_q[\alpha_j])$, it follows that the set
	$\mathcal{P}_j \left( \mathcal{B}(A_j) \right)$ is linear independent. So, since $\dim_{\mathbb{F}_q}(A_j) = 2 \deg(f_j) = |\mathcal{P}_j \left( \mathcal{B}(A_j) \right)|$, it follows that $\mathcal{B}( A_j )$ is a $\mathbb{F}_q$--basis of $\overline{ A_j }$. Claim 1) is proved.	
	
	Let us prove claim 2). Using (\ref{eq:Ij}), we get
	\[ 
	I_j\left(u(\alpha_j + \alpha_j^{-1}), v(\alpha_j + \alpha_j^{-1}) \right) =
	A_j \begin{pmatrix}
	v(\alpha_j + \alpha_j^{-1}) & -u(\alpha_j + \alpha_j^{-1}) \\
	0 & 0		
	\end{pmatrix}.
	\]
	Using Theorem \ref{theor:3}, we obtain
	\[ 
	\epsilon_j 
	\begin{pmatrix}
	v(\alpha_j + \alpha_j^{-1}) & -u(\alpha_j + \alpha_j^{-1}) \\
	0 & 0		
	\end{pmatrix} =
	(1+b) F_{j,1}(a)G_{j,1}(a) \left( au(a+a^{-1}) - v(a+a^{-1}) \right) e_{f_j}(a).
	\]
	Let $T = (1+b) F_{j,1}(a)G_{j,1}(a) \left( au(a+a^{-1}) - v(a+a^{-1}) \right) e_{f_j}(a)$; by definition we have
	\[ 
	\mathcal{B}\left( I_j\left(u(\alpha_j + \alpha_j^{-1}), v(\alpha_j + \alpha_j^{-1})\right) \right) =
	\left\{ a^i T \mid i = 0 \dots \deg(f_j)-1 \right\}
	\subset \overline{ I_j\left(u(\alpha_j + \alpha_j^{-1}), v(\alpha_j + \alpha_j^{-1})\right) }.
	\]
	We show that this set is linear independent. Assume the converse. Then there exists a polynomial $P(x) \in \mathbb{F}_q[x]$, $\deg(P) < \deg(f_j)$, such that
	$P(a) T = 0$. It follows that
	\[ 
	0  = \sigma^{-1} \mathcal{P}_j(P(a) T) = \tau_j (P(a) T) =
	\begin{pmatrix}
	P(\alpha_j) T(\alpha_j) & P(\alpha_j) T(\alpha_j^{-1}) \\
	P(\alpha_j^{-1}) T(\alpha_j) & P(\alpha_j^{-1}) T(\alpha_j^{-1})
	\end{pmatrix}
	\]
	(see (\ref{eq:2})). Since  $f_j = f_j^*$ is the minimal polynomial of $\alpha_j, \alpha_j^{-1} \in \mathbb{F}_q[\alpha_j]$, and $\deg(P) < \deg(f_j)$, it follows that $P(\alpha_j) \neq 0$, $P(\alpha_j^{-1}) \neq 0$. So, we get $T(\alpha_j) = T(\alpha_j^{-1}) = 0$. Since $(v(\alpha_j+\alpha_j^{-1}), u(\alpha_j+\alpha_j^{-1})) \neq (0, 0)$ and
	\[ T(\alpha_j) = -\frac{\alpha_j v(\alpha_j + \alpha_j^{-1}) - u(\alpha_j + \alpha_j^{-1})}{(\alpha_j^2 - 1)}. \]
	we have arrived at a contradiction.	Hence the set 
	$\mathcal{B}\left( I_j\left(u(\alpha_j + \alpha_j^{-1}), v(\alpha_j + \alpha_j^{-1})\right) \right)$ is linear independent and its cardinality is equal to $\deg(f_j) = \dim_{\mathbb{F}_q}\left(I_j\left(u(\alpha_j + \alpha_j^{-1}), v(\alpha_j + \alpha_j^{-1})\right)  \right)$. So, claim 2) is proved.
\end{proof}
\begin{remark} \label{remark:7} \hspace{2em}
	\begin{itemize} 
		\item[1)] Let $P \in \mathbb{F}_q[x]$ be a polynomial such that $P(\alpha_j) \neq 0$, $P(\alpha_j^{-1}) \neq 0$. Then the matrix
		\[ \mathcal{P}_j(P(a)) = Z_j^{-1} \begin{pmatrix}
		P(\alpha_j) & 0 \\
		0 & P(\alpha_j^{-1})
		\end{pmatrix} Z_j \in A_j  \]
		is invertible, and the set $\mathcal{B}(A_j)P(a)$ is also a $\mathbb{F}_q$--basis of $\overline{A_j}$.
		\item[2)]  Let $u, v \in \mathbb{F}_q[x]$ be polynomials such that $(v(\alpha_j+\alpha_j^{-1}), u(\alpha_j+\alpha_j^{-1})) \neq (0, 0)$. Suppose that
		\[ P(a) = F_{j,1}(a)G_{j,1}(a) \left( au(a+a^{-1}) - v(a+a^{-1}) \right); \]
		 then $P(\alpha_j) \neq 0$, $P(\alpha_j^{-1}) \neq 0$, and Lemma \ref{lem:10} implies that any element of $\mathcal{B}\left( I_j\left(u(\alpha_j + \alpha_j^{-1}), v(\alpha_j + \alpha_j^{-1})\right) \right)$ is a sum of some elements of $\mathcal{B}(A_j)P(a)$.		
	\end{itemize}	
\end{remark}
Lemma \ref{lem:10} allows us to describe generating matrices of $\overline{A_j}$ and $\overline{I_j\left(u(\alpha_j + \alpha_j^{-1}), v(\alpha_j + \alpha_j^{-1})\right)}$.
\begin{corollary} \label{col:5} Let $\zeta(n) + 1 \leq j \leq r$. Then
	\[ \textbf{G}_{\overline{A_j}} = 
	 	\begin{pmatrix}
	 	\phi(e_{f_j}(a)) \\ \phi(be_{f_j}(a)) \\
	 	\phi(ae_{f_j}(a)) \\ \phi(ab e_{f_j}(a)) \\ \vdots \\
	 	\phi(a^{\deg(f_j) - 1} e_{f_j}(a)) \\ \phi(a^{\deg(f_j) - 1} b e_{f_j}(a))
	 	\end{pmatrix} =
		\begin{pmatrix}
		\psi(e_{f_j}(a)) & 0 \\ 
		0 & \psi(e_{f_j}(a)) \\
		\psi(a e_{f_j}(a)) & 0 \\ 
		0 & \psi(a^{-1} e_{f_j}(a)) \\
		\vdots & \vdots \\
		\psi(a^{\deg(f_j) - 1} e_{f_j}(a)) & 0 \\ 
		0 & \psi(a^{-\deg(f_j) + 1} e_{f_j}(a)) \\
		\end{pmatrix};
	\]	
	\[ 
		\textbf{G}_{\overline{I_j\left(u(\alpha_j + \alpha_j^{-1}), v(\alpha_j + \alpha_j^{-1})\right)}} =
		\begin{pmatrix}
		\psi( P(a) ) & \psi( P(a) ) \\
		\psi( a P(a) ) & \psi( a^{-1} P(a) ) \\
		\vdots & \vdots \\
		\psi(a^{\deg(f_j) - 1} P(a)) & \psi(a^{-\deg(f_j) + 1} P(a))
		\end{pmatrix}
	 \]
	 (see (\ref{eq:phipsi})), where $P(a) = F_{j,1}(a)G_{j,1}(a) \left( au(a+a^{-1}) - v(a+a^{-1}) \right) e_{f_j}(a)$.
\end{corollary}

Finally, let us consider the case $r + 1 \leq j \leq r+s$.
\begin{Lemma} \label{lem:8}
	Let $\gcd(q,n)=1$, and let $r+1 \leq j \leq r+s$. Then
	\begin{enumerate}
		\item[1)] the set
		\[ 
			\mathcal{B}(A_j) := \left\{ 		
			a^i b^k e_{f_j}(a), \;
			a^i b^k e_{f^*_j}(a), \;
			 \mid i = 0 \dots \deg(f_j) - 1, \; k=0,1
			\right\}
		\]
		is a $\mathbb{F}_q$--basis of $\overline{ A_j }$;
		\item[2)]  for any polynomials $u, v \in \mathbb{F}_q[x]$ such that $(v(\alpha_j^{-1}), u(\alpha_j)) \neq (0, 0)$, the set
		\[
			\mathcal{B}(I_j(-v(\alpha_j^{-1}), u(\alpha_j))) :=
			\left\{
			a^i b^k \left( u(a) e_{f_j}(a) + bv(a) e_{f^*_j}(a) \right) \mid
			i = 0 \dots \deg(f_j) - 1, \; k = 0, 1
			\right\}
		\]
		is a $\mathbb{F}_q$--basis of $\overline{I_j(-v(\alpha_j^{-1}), u(\alpha_j))})$.
	\end{enumerate}
\end{Lemma}
\begin{proof}
	Recall that $A_j = \mathrm{M}_2(\mathbb{F}_q[\alpha_j])$ for $r+1 \leq j \leq r+s$. 	
	Let us prove claim 1). Since the sets
	\[ 
		\{ \alpha_j^{i} \mid i=0\dots \deg(f_j)-1 \}, \quad
		\{ \alpha_j^{-i} \mid i=0\dots \deg(f_j)-1 \}
	\]
	are $\mathbb{F}_q$--bases of the $\mathbb{F}_q[\alpha_j]$ (see section 2), it follows that the set	
	\[ T = \left\{		
	\begin{pmatrix}
	\alpha_j^i & 0 \\
	0 & 0		
	\end{pmatrix},
	\begin{pmatrix}
	0 & 0 \\
	0 & \alpha_j^{-i}		
	\end{pmatrix} 
	\begin{pmatrix}
	0 & 0 \\
	\alpha_j^{-i} & 0		
	\end{pmatrix},
	\begin{pmatrix}
	0 & \alpha_j^{i} \\
	0 & 0		
	\end{pmatrix},
	\Big | \;
	i = 1, \dots ,\deg(f_j) - 1
	\right\}. \]
	is a $\mathbb{F}_q$--basis of $A_j$. Using Theorem \ref{theor:3}, we see that $\epsilon_j(T) = \mathcal{B}(A_j)$. Hence $\mathcal{B}(A_j)$  is a basis of $\overline{ A_j }$. Claim 1) is proved.

	Let us prove claim 2). The set  
	\[ T' = \left\{		
	\alpha_j^i \begin{pmatrix}
	u(\alpha_j) & v(\alpha_j^{-1}) \\
	0 & 0		
	\end{pmatrix},
	\alpha_j^{-i} \begin{pmatrix}
	0 & 0 \\
	u(\alpha_j) & v(\alpha_j^{-1}) \\
	\end{pmatrix}
	\Big | \;
	i = 0...\dots \deg(f_j) 
	\right\} \]
	is a $\mathbb{F}_q$--basis of $I_j(-v(\alpha_j^{-1}), u(\alpha_j))$ (see (\ref{eq:a1})). Since
	\[  
	\epsilon_j \begin{pmatrix}
	u(\alpha_j) & v(\alpha_j^{-1}) \\
	0 & 0		
	\end{pmatrix}  =		
	u(a) e_{f_j}(a) + bv(a) e_{f^*_j}(a), \quad 
 	\mathcal{P}_j(a^i) = \begin{pmatrix}
 	\alpha_j^i & 0 \\ 0 & \alpha_j^{-i}
 	\end{pmatrix}, \quad
	 \mathcal{P}_j(b) = \begin{pmatrix}
 	0 & 1 \\ 1 & 0
 	\end{pmatrix}
 	\]
	 (see Theorem \ref{theor:3}, (\ref{eq:a1})), it follows that
	 $\mathcal{B}(I_j(-v(\alpha_j^{-1}), u(\alpha_j)))  = \epsilon_j(T')$. Hence  $\mathcal{B}(I_j(-v(\alpha_j^{-1}), u(\alpha_j)))$  is a $\mathbb{F}_q$--basis of $\overline{I_j(-v(\alpha_j^{-1}), u(\alpha_j))})$. Claim 2) is proved.
\end{proof}
\begin{remark} \label{remark:8}
	Let $u, v \in \mathbb{F}_q[x]$ be polynomials such that $u(\alpha_j) \neq 0$ and $v(\alpha_j^{-1}) \neq 0$. Then	
	\begin{itemize}
		\item[1)] the element
		\[ 
		\mathcal{P}_j \left( u(a)e_{f_j}(a) + v(a)e_{f_j^*}(a) \right) =
		\begin{pmatrix}
		u(\alpha_j) & 0 \\ 0 & v(\alpha_j^{-1})
		\end{pmatrix},
		\]
		(see (\ref{eq:2}), (\ref{eq:orth})) is invertible in $A_j$. In addition, since $\mathcal{P}_j$ is an isomorphism between $A_j$ and $\overline{A_j}$, it follows that the set
		\[  \begin{split}
			\mathcal{B}( A_j ) \left( u(a)e_{f_j}(a) + v(a)e_{f_j^*}(a) \right) = \\ =
			\left\{ 		
			a^i b^k u(a) e_{f_j}(a), \;
			a^i b^k v(a) e_{f^*_j}(a), \;
			\mid i = 0 \dots \deg(f_j) - 1, \; k=0,1
			\right\}	
		\end{split} \]
		is a basis of $\overline{ A_j }$;
		\item[2)] since $
			a^i b^k \left( u(a) e_{f_j}(a) + bv(a) e_{f^*_j}(a) \right) = ( a^i b^k u(a) e_{f_j}(a) ) + ( a^i b^{k+1} v(a) e_{f^*_j}(a) ),
		$
		it follows that any element of $\mathcal{B}\left( \overline{I_j(-v(\alpha_j^{-1}), u(\alpha_j))} \right)$ is a sum of elements of $\mathcal{B}( A_j ) \left( u(a)e_{f_j}(a) + v(a)e_{f_j^*}(a) \right)$.
	\end{itemize}
\end{remark}
Lemma \ref{lem:8} also allows us to explicitly describe generating matrices of $\overline{A_j}$ and $\overline{I(-v(\alpha_j^{-1}), u(\alpha_j))}$ for $r+1 \leq j \leq r+s$.
\begin{corollary} \label{col:7}
	Let $r+1 \leq j \leq r+s$. Then for any polynomials $u, v \in \mathbb{F}_q[x]$ such that $u(\alpha_j) \neq 0$, $v(\alpha_j^{-1}) \neq 0$ the following formulas hold (see (\ref{eq:phipsi}))
	\[
		\textbf{G}_{\overline{A_j}} =
		\begin{pmatrix}
		\vdots \\
		\phi(a^i e_{f_j}(a)) \\
		\phi(a^i b e_{f^*_j}(a)) \\ 
		\phi(a^i e_{f^*_j}(a)) \\
		\phi(a^i b e_{f_j}(a)) \\ 
		\vdots
		\end{pmatrix} =
		\begin{pmatrix}
		\vdots & \vdots \\
		\psi(a^i e_{f_j}(a)) & 0 \\
		0 & \psi(a^{-i} e_{f^*_j}(a)) \\
		\psi(a^i e_{f^*_j}(a)) & 0 \\
		0 & \psi(a^{-i} e_{f_j}(a)) \\
		\vdots & \vdots
		\end{pmatrix} \quad (i = 0, \dots, \deg(f_j) - 1);
	\]
	\[ 
		\textbf{G}_{\overline{I_j(-v(\alpha_j^{-1}), u(\alpha_j))})} = 
		\begin{pmatrix}
		\vdots & \vdots \\
		\psi(a^{i} u(a) e_{f_j}(a)) & \psi(a^{-i} v(a) e_{f^*_j}(a)) \\
		\psi(a^{i} v(a) e_{f^*_j}(a)) & \psi(a^{-i} u(a) e_{f_j}(a)) \\
		\vdots & \vdots 
		\end{pmatrix} \quad (i = 0, \dots, \deg(f_j) - 1);
	\]
	\[ 
		\textbf{G}_{\overline{ I_j(1, 0) }} = 
		\begin{pmatrix}
		\vdots & \vdots \\
		\psi(a^{i} e_{f^*_j}(a)) & 0 \\
		0 & \psi(a^{-i} e_{f^*_j}(a)) \\
		\vdots & \vdots 
		\end{pmatrix} \quad (i = 0, \dots, \deg(f_j) - 1);
	\]
	\[ 
		\textbf{G}_{\overline{ I_j(0, 1) }} = 
		\begin{pmatrix}
		\vdots & \vdots \\
		\psi(a^{i} e_{f_j}(a)) & 0 \\
		0 & \psi(a^{-i} e_{f_j}(a)) \\
		\vdots & \vdots 
		\end{pmatrix} \quad (i = 0, \dots, \deg(f_j) - 1).
	\]
\end{corollary}

Now we can describe a $\mathbb{F}_q$--basis of an arbitary dihedral code $I \subset \mathbb{F}_qD_{n}$. In addition, we can obtain its  generating and check matrices.

\begin{theorem} \label{theor:matr}
	Let $\gcd(q, n) = 1$, let $I \subset \mathbb{F}_qD_{n}$ be a dihedral code, and let
	\[
	\mathcal{P}(I) = \bigoplus_{j=1}^{r+s} B_j, \quad	
	B_j = \begin{cases}
	A_j, & j \in J_1 \\
	I_j(0, 1), & j \in J_2 \\
	I_j(1, 0), & j \in J_3 \\
	I_j(-x_j, 1), & j \in J_4 \\			
	0, & j \notin J_1 \cup J_2	\cup J_3 \cup J_4
	\end{cases}
	\]
	(see (\ref{eq:form}) -- (\ref{eq:15})). Then
	\begin{enumerate}
		\item[1)] the set $T = \bigcup_{j=1}^{r+s} \mathcal{B}(B_j)$ is a basis of $I$;
		\item[2)] the block matrix $\textbf{G}_I$, given by
		\[ \textbf{G}_I =  
			\begin{pmatrix}
				\vdots \\
				\textbf{G}_{\overline{B_j}} \\
				\vdots
			\end{pmatrix}, \quad j \in J_1 \cup J_2 \cup J_3 \cup J_4,
		\]
		is a generating matrix of $I$.
	\end{enumerate}
\end{theorem}
\begin{proof}
	Decomposition (\ref{eq:2}) and Theorem \ref{theor:codes} imply that  $I$ is a direct sum of $\overline{B_j}$, $j \in J_1 \cup J_2 \cup J_3 \cup J_4$. It follows that the union of bases of the left ideals $\overline{B_j}$ is a basis of $I$. In fact, bases of the left ideals $\overline{B_j}$ were described in Lemmas \ref{lem:7}--\ref{lem:8}. This concludes the proof of claim 1).
	
	Claim 2) is obtained from claim 1) and corollaries \ref{col:3}--\ref{col:7}. 
\end{proof}
\begin{remark}
	As is well--known, a check matrix $H$ of a code $I$ is a generating matrix of the dual code $I^{\perp}$. Hence  $H$ can be obtained by using Theorem \ref{theor:dual} and Theorem \ref{theor:matr} consistently.
\end{remark}
\begin{corollary} \label{col:matr}
	Let $I$ be a dihedral code. Then by rearranging the rows its generating matrix $\textbf{G}_I$ can be represented as
	\[
	\left( \begin{array}{c|c}
	G_{I, 1} & 0 \\
	0 & G_{I, 2} \\
	\cline{1-2}
	G_{I, 3} & G_{I, 4}
	\end{array}	\right),
	\]
	where
	\begin{itemize}
		\item[1)] the matrices $G_{I, 1}$, $G_{I, 2}$ are generating matrices of some cyclic code, i.e. $G_{I, 2} = S_1 G_{I, 1}$, where $S_1$ is  invertible;
		\item[2)] the matrices $G_{I, 3}$, $G_{I, 4}$ are generating matrices of some cyclic code, i.e. $G_{I, 4} = S_2 G_{I, 3}$,  where $S_2$ is  invertible.
	\end{itemize}	
\end{corollary}
\begin{proof}
	Indeed, corollaries \ref{col:3}, \ref{col:5}, \ref{col:7} imply that for any $j$ the matrix $\textbf{G}_{\overline{B_j}}$  can be represented as
	\[ 
		\left( \begin{array}{c|c}
		T_{j, 1} & 0 \\ 0 & T_{j, 2}
		 \end{array}  \right), \quad
		 \left( \begin{array}{c|c}
		 T_{j, 1} & T_{j, 2}
		 \end{array}  \right).
	\]
	by rearranging the rows. In addition, since $T_{j, 2}$ can be obtained from $T_{j, 1}$ by rearranging the rows, by using cyclic shifts of the rows (see corollaries \ref{col:5}, \ref{col:7}), and by multiplying the rows by the elements of $\mathbb{F}_q$, it follows that $T_{j, 2}$, $T_{j, 2}$ are generating matrices of some cyclic code. This concludes the proof of the corollary.
\end{proof}

\section{Exterior and interior induced codes} \label{sec:ind}
Let $G$ be a finite group, let $H$ be a subgroup of $G$.  Given a $H$--code, there is a natural way to construct a $G$--code. Let $I$ be a left ideal of $\mathbb{F}_qH$, the $G$--code $I^G := (\mathbb{F}_qG)I$ is called\textit{ $H$--induced} \cite{zimmerman, DeuKos15}.

Let $\mathcal{T}$ be a right transveral for $H$ in $G$. Recall, $\mathcal{T}$ is a right transversal for $H$ in $G$ if and only if every right coset of $H$ contains exactly one element of $\mathcal T$. In \cite{zimmerman, DeuKos15} it was proved that if $I$ is an $[|H|, k, d]$--code then $J$ is an $[|G|, |\mathcal T| k, d]$--code. In addition, if a set $S$ is a basis of $I$ then $\mathcal T S$ is a basis of $J$. 
Let $i \in \mathbb{F}_qH$ be a principal element if $I$, i.e. $I = (\mathbb{F}_qH)i$. Since $J = (\mathbb{F}_qG)I = (\mathbb{F}_qG)(\mathbb{F}_qH)i$, it follows that $i$ is also a principal element of $J$.

By $\langle a \rangle_{D_{n}}$ we denote a cyclic subgroup of $D_{n}$ generated by the element $a \in D_{n}$. Below, by induced codes we mean $\langle a \rangle_{D_{n}}$--induced $D_n$--codes. In this section, given a $D_n$--code, we consider two important induced codes and describe their properties. It allows us to establish the connection between cyclic and dihedral codes.

In \cite{VedDeu18}, theorem 5, the $\mathbb{F}_q$--linear map $\mathtt{pr}_a: \mathbb{F}_qD_{n} \to \mathbb{F}_q \langle a \rangle_{D_{n}}$ was defined as
\[ \mathtt{pr}_a: P(a) + bQ(a) \mapsto P(a).  \]
Note that for an element $t = P(a) + bQ(a) \in \mathbb{F}_qD_{n}$ the following formula holds
\begin{equation} \label{eq:pr}
	t = \mathtt{pr}_a(t) + b\mathtt{pr}_a(bt).
\end{equation}
In \cite{VedDeu18} it was proved that for any code $C \subset \mathbb{F}_qD_{n}$ the image $\mathtt{pr}_a(C)$ is an ideal of $\mathbb{F}_q\langle a \rangle_{D_{n}}$, i.e. a cyclic code.

Define the exterior induced code $C_{ext}$ as
\[ C_{ext} := \left(\mathtt{pr}_a(C)\right)^{D_n} = \mathbb{F}_qD_{n} \mathtt{pr}_a(C)  \]
In \cite{VedDeu18} it was also proved that $C \subset C_{ext}$. 
\begin{Lemma}
	Let $C \subset \mathbb{F}_qD_{n}$ be a code, let $T$ be an ideal of $\mathbb{F}_q\langle a \rangle_{D_{n}}$. If $C \subset (\mathbb{F}_qD_{n})T$, then $\mathtt{pr}_a(C) \subset T$ and $C_{ext} \subset (\mathbb{F}_qD_{n})T$.
\end{Lemma}
\begin{proof}
	Using
	\[ (\mathbb{F}_qD_{n})T = \left\{ P(a)R(a) + bQ(a) R(a) \mid P(a) + bQ(a) \in \mathbb{F}_qD_{n}, \; R(a) \in T  \right\}, \] 
	we see that $\mathtt{pr}_a((\mathbb{F}_qD_{n})T) = T$. Hence $C \subset (\mathbb{F}_qD_{n})T$ implies that $\mathtt{pr}_a(C) \subset T$. Further, since $C_{ext}=\mathbb{F}_qD_{n} \mathtt{pr}_a(C)$, it follows that $C_{ext} \subset (\mathbb{F}_qD_{n})T$.
\end{proof}

Hence $C_{ext}$ is the smallest $\langle a \rangle_{D_{n}}$--induced $D_n$--code that contains $C$. \textit{$C_{ext}$ is called the exterior induced code of $C$.} 

In \cite{VedDeu18}, p. 240, for given a code $C \subset \mathbb{F}_qD_{n}$, the code $C_{int} := \mathbb{F}_qD_{n}(C \cap \mathbb{F}_q\langle a \rangle_{D_{n}})$ was defined. It was also shown that $C_{int} \subset C$. Clearly, $C_{int}$ is the largest $\langle a \rangle_{D_{n}}$--induced $D_n$--code that is contained in $C$. \textit{$C_{int}$ is called the interior induced code of $C$.}  

So, we have
\begin{equation} \label{eq:chain}
C_{int} \subset C \subset C_{ext}.
\end{equation}
Note that $C \subset \mathbb{F}_qD_{n}$ is $\langle a \rangle_{D_{n}}$--induced if and only if $C_{int} = C = C_{ext}$.

In the following theorem we describe the image of an exterior code under (\ref{eq:2}). To do this, we need the following lemma.
\begin{Lemma} \label{lem:11}
	Let $C \subset \mathbb{F}_qD_{n}$ be a dihedral code generated by an element $u = P(a) + bQ(a)$, i.e. $C = (\mathbb{F}_qD_{n})u$. Then
	\[ 
		C_{ext} = \{ t_1 P(a) + t_2 Q(a) \mid t_1, t_2 \in \mathbb{F}_qD_{n}  \}. 
	\]
	In particular, if $Q(a) = 0$, then $C_{ext} = C$.
\end{Lemma}
\begin{proof}
	Using (\ref{eq:pq}), we get
	\[ \begin{split}
	C = (\mathbb{F}_qD_{n})u = \\ =
	\left\{ \left( P_1(a)P(a) + Q_1(a^{-1})Q(a) \right) + 
	b \left( P_1(a^{-1})Q(a) + Q_1(a)P(a) \right) \mid
	P_1(a) + b Q_1(a) \in \mathbb{F}_qD_{n}  \right\}.
	\end{split}
	\]
	It follows that
	\[  \mathtt{pr}_a(C) = 
	\left\{ P_1(a)P(a) + Q_1(a^{-1})Q(a) \mid
	P_1(a), Q_1(a^{-1}) \in \mathbb{F}_q \langle a \rangle_{D_{n}}  \right\}.
	\]
	Hence $P(a), Q(a) \in \mathtt{pr}_a(C)$. This implies that
	$C_{ext} = (\mathbb{F}_qD_{n})\mathtt{pr}_a(C) = \{ t_1 P(a) + t_2 Q(a) \mid t_1, t_2 \in \mathbb{F}_qD_{n}  \}$.
	The lemma is proved.
\end{proof}

\begin{theorem} \label{theor:8}
	Let $\gcd(q, n) = 1$, let $C \subset \mathbb{F}_qD_{n}$ be a dihedral code, and let
	\[
	\mathcal{P}(C) = \bigoplus_{j=1}^{r+s} B_j, \quad
	B_j = \begin{cases}
	A_j, & j \in J_1 \\
	I_j(0, 1), & j \in J_2 \\
	I_j(1, 0), & j \in J_3 \\
	I_j(-x_j, 1) & j \in J_4 \\			
	0, & j \notin J_1 \cup J_2	\cup J_3 \cup J_4
	\end{cases},
	\]
	(see (\ref{eq:form}) -- (\ref{eq:15})). Then
	\[
	\mathcal{P}(C_{ext}) = \bigoplus_{j=1}^{r+s} B_j^{ext}, \quad
	B_j^{ext} = \begin{cases}
	A_j, & j \in K_1 \\
	I_j(0, 1) & j \in K_2 \\
	I_j(1, 0), & j \in K_3 \\		
	0, & j \notin K_1 \cup K_2	\cup K_3
	\end{cases},
	\]
	where
	\[ K_1 = J_1 \cup J_4 \cup \{ j \in J_2 \cup J_3 \mid 1 \leq j \leq r \}, \]
	\[ K_2 = \{ j \in J_2 \mid r+1 \leq j \leq r+s \}, \]
	\[ K_3 = \{ j \in J_3 \mid r+1 \leq j \leq r+s \}. \]
\end{theorem}
\begin{proof}
	Let us consider several cases. 
	
	Case 1: $1 \leq j \leq \zeta(n)$.  Using Lemma \ref{lem:7}, we obtain
	\[ \overline{ A_j } = (\mathbb{F}_qD_{n})e_{f_j}(a), \quad
	\overline{ I_j(1, 0) } = (\mathbb{F}_qD_{n})\left( e_{f_j}(a) + b e_{f_j}(a) \right), \quad
	\overline{ I_j(0, 1) } = (\mathbb{F}_qD_{n})\left( e_{f_j}(a) - be_{f_j}(a) \right).	 	
	\]
	Consequentially, using Lemma \ref{lem:11}, we find that
	\[ 
	\left(\overline{ A_j }\right)_{ext} = (\mathbb{F}_qD_{n})e_{f_j}(a) = \overline{A_j},
	\quad
	\left(\overline{ I_j(1, 0) }\right)_{ext} = \overline{A_j}, \quad
	\left(\overline{ I_j(0, 1) }\right)_{ext} = \overline{A_j}.
	\]
	
	Case 2: $\zeta(n) + 1 \leq j \leq r$. Lemma \ref{lem:10} implies that
	\[ \begin{split}
		\overline{ A_j } = (\mathbb{F}_qD_{n})e_{f_j}(a), \quad
	   \overline{I_j\left(u(\alpha_j + \alpha_j^{-1}), v(\alpha_j + \alpha_j^{-1})\right)} = (\mathbb{F}_qD_{n})(1 + b) P(a), \\
	   P(a) = F_{j,1}(a)G_{j,1}(a) \left( au(a+a^{-1}) - v(a+a^{-1}) \right) e_{f_j}(a),
	   \end{split}
	 \]
	 for any polynomials $u, v \in \mathbb{F}_q[x]$ such that $\left( u(\alpha_j + \alpha_j^{-1}), v(\alpha_j + \alpha_j^{-1}) \right) \neq (0, 0)$.
	Hence using Lemma \ref{lem:11}, we obtain
	\[ \left( \overline{ A_j }\right)_{ext} = (\mathbb{F}_qD_{n})e_{f_j}(a) = \overline{A_j}, \quad
	 \left(\overline{I_j\left(u(\alpha_j + \alpha_j^{-1}), v(\alpha_j + \alpha_j^{-1})\right)}\right)_{ext} = (\mathbb{F}_qD_{n}) P(a) = \overline{A_j},
	\]
	(see Remark \ref{remark:7}). In particular,
	$
		\left(\overline{I_j\left(0, 1 \right)}\right)_{ext} = \left(\overline{I_j\left(1, 0 \right)}\right)_{ext} = 
		\left(\overline{I_j\left(-x_j, 1 \right)}\right)_{ext} =
		\overline{A_j}.
	$
	
	Case 3: $r+1 \leq j \leq r+s$. Lemma \ref{lem:8} implies that 
	\[ \overline{ A_j } = (\mathbb{F}_qD_{n}) \left( e_{f_j}(a) + e_{f_j^*}(a) \right), \quad
		\overline{I_j(-v(\alpha_j^{-1}), u(\alpha_j))})	=
		(\mathbb{F}_qD_{n}) \left( u(a)e_{f_j}(a) + bv(a)e_{f_j^*}(a) \right)
	\] 
	for any polynomials $u, v \in \mathbb{F}_q[x]$ such that $\left( u(\alpha_j), v(\alpha_j^{-1}) \right) \neq (0, 0)$. So, using Lemma \ref{lem:11}, we obtain
	\[ 
		\left( \overline{ A_j }\right)_{ext} = (\mathbb{F}_qD_{n}) \left( e_{f_j}(a) + e_{f_j^*}(a) \right) = \overline{ A_j },
	\] \[
		\left( \overline{I_j(-v(\alpha_j^{-1}), u(\alpha_j))} \right)_{ext} = 
	\{ t_1 u(a) e_{f_j}(a) + t_2 v(a) e_{f_j^*}(a) \mid t_1, t_2 \in \mathbb{F}_qD_{n}  \}, \]
	Hence
	\[ \left( \overline{I_j(1, 0)}\right)_{ext} = \overline{I_j(1, 0)}, \quad
		\left( \overline{I_j(0, 1)}\right)_{ext} = \overline{I_j(0, 1)}, \quad 
		\left( \overline{I_j(-x_j, 1)} \right)_{ext} = A_j,
	\]
	(see Remark \ref{remark:8}).
	
	Therefore, since $C$ is a direct sum of $\overline{B_j}$, the theorem is proved. 
\end{proof}

Now we consider the interior induced codes and the their images under (\ref{eq:2}).
\begin{theorem} \label{theor:9}
	Let $\gcd(q, n) = 1$, let $C \subset \mathbb{F}_qD_{n}$ be a dihedral code, and let
	\[
	\mathcal{P}(C) = \bigoplus_{j=1}^{r+s} B_j, \quad
	B_j = \begin{cases}
	A_j, & j \in J_1 \\
	I_j(0, 1), & j \in J_2 \\
	I_j(1, 0), & j \in J_3 \\
	I_j(-x_j, 1) & j \in J_4 \\			
	0, & j \notin J_1 \cup J_2	\cup J_3 \cup J_4
	\end{cases}
	\]
	(see (\ref{eq:form}) -- (\ref{eq:15})). Then
	\[
	\mathcal{P}(C_{int}) = \bigoplus_{j=1}^{r+s} B_j^{int}, \quad
	B_j^{int} = \begin{cases}
	A_j, & j \in K_1 \\
	I_j(0, 1) & j \in K_2 \\
	I_j(1, 0), & j \in K_3 \\		
	0, & j \notin K_1 \cup K_2	\cup K_3
	\end{cases},
	\]
	where
	\[ K_1 = J_1, \quad
	 K_2 = \{ j \in J_2 \mid r+1 \leq j \leq r+s \}, \quad
	 K_3 = \{ j \in J_3 \mid r+1 \leq j \leq r+s \}. \]
\end{theorem}
\begin{proof}
	Using (\ref{eq:2}), we see that
	\[ \mathcal{P}(C_{int}) = \mathcal{P}\left((\mathbb{F}_qD_{n})(\mathbb{F}_q\langle a \rangle_{D_{n}} \cap C) \right) =
	\mathcal{P}(\mathbb{F}_qD_{n})\left( 
	\mathcal{P}(\mathbb{F}_q \langle a \rangle_{D_{n}}) \cap \mathcal{P}(C)
	\right) =
	\bigoplus_{j=1}^{r+s} A_j \left( \mathcal{P}_j(\mathbb{F}_q \langle a \rangle_{D_{n}}) \cap B_j \right).
	\]	
	To prove the theorem, it remains to find $\mathcal{P}_j(\mathbb{F}_q \langle a \rangle_{D_{n}}) \cap B_j$ and $A_j \left( \mathcal{P}_j(\mathbb{F}_q \langle a \rangle_{D_{n}}) \cap B_j \right)$. 
	
	If $B_j=A_j$ or $B_j = 0$, then we obviously have
	\[
	 	\mathcal{P}_j(\mathbb{F}_q\langle a \rangle_{D_{n}}) \cap A_j = \mathcal{P}_j(\mathbb{F}_q\langle a \rangle_{D_{n}}),	\quad 
	 	\mathcal{P}_j(\mathbb{F}_q\langle a \rangle_{D_{n}}) \cap 0 = 0.	 	
	\]
	Since $1_{A_j} \in \mathcal{P}_j(\mathbb{F}_q\langle a \rangle_{D_{n}})$, it follows that
	\[
	A_j \left(\mathcal{P}_j(\mathbb{F}_q\langle a \rangle_{D_{n}}) \cap A_j \right) = A_j,	\quad 
	A_j \left(\mathcal{P}_j(\mathbb{F}_q\langle a \rangle_{D_{n}}) \cap 0 \right) = 0.
	\]
	
	Not let $B_j \neq A_j$, $B_j \neq 0$. In the proof we consider three cases.
	
	Case 1:  $1 \leq j \leq \zeta(n)$. By definition (\ref{eq:2}), we have
	$ \mathcal{P}_j(\mathbb{F}_q\langle a \rangle_{D_{n}}) = \tau_j (\mathbb{F}_q\langle a \rangle_{D_{n}}) = \{ t + 0h \mid t \in \mathbb{F}_q \}$. So, using (\ref{eq:a1}), we get
	\[ 
		\mathcal{P}_j(\mathbb{F}_q\langle a \rangle_{D_{n}}) \cap I_j(1, 0) = 
		\{ t + 0h \mid t \in \mathbb{F}_q \} \cap
		\{ t + th \mid t \in \mathbb{F}_q \} = 0.
	\]
	Similarly, $\mathcal{P}_j(\mathbb{F}_q\langle a \rangle_{D_{n}}) \cap I_j(0, 1) = 0$.
	
	Case 2: $\zeta(n) + 1 \leq j \leq r$. Lemma \ref{lem:10} implies that for any $u(x), v(x) \in \mathbb{F}_q[x]$ such that
	\[ (v(\alpha_j+\alpha_j^{-1}), u(\alpha_j+\alpha_j^{-1})) \neq (0, 0),\] 
	an arbitrary element $w$ of $\overline{I_j\left(u(\alpha_j + \alpha_j^{-1}), v(\alpha_j + \alpha_j^{-1})\right)}$ can be represented as
	\[ w = t(a)(1+b)T(a) = t(a)T(a) + bt(a^{-1})T(a), \quad	
		T(a) = F_{j,1}(a)G_{j,1}(a) \left( au(a+a^{-1}) - v(a+a^{-1}) \right) e_{f_j}(a),	
	\]
	where $t(x) \in \mathbb{F}_q[x]$, $\deg(t) < \deg(f_j)$. Assuming that $w \in \mathbb{F}_q\langle a \rangle_{D_{n}}$, we obtain $t(a^{-1})T(a) = 0$. So, since $T(\alpha_j) \neq 0$ (see Remark \ref{remark:7}), it follows that $t(\alpha_j^{-1}) = 0$. Since $f_j(x) = f_j^*(x)$ is the minimal polynomial of $\alpha_j$ and $\alpha_j^{-1}$, we conclude that $t(\alpha_j) = 0$.
	Using the definition of $\tau_j$ (see (\ref{eq:2})), we get
	\[ \tau_j(w) = \begin{pmatrix} 
		t(\alpha_j)T(\alpha_j) &  t(\alpha_j)T(\alpha_j^{-1}) \\
		t(\alpha_j^{-1})T(\alpha_j) & t(\alpha_j^{-1})T(\alpha_j^{-1})
		\end{pmatrix} =	0 \]
	So, since $\mathcal{P}_j = \sigma_j \tau_j$, it follows that $\mathcal{P}_j(\mathbb{F}_q\langle a \rangle_{D_{n}}) \cap I_j\left(u(\alpha_j + \alpha_j^{-1}), v(\alpha_j + \alpha_j^{-1})\right) = 0$.

	Case 3: $r + 1 \leq j \leq r+s$. Using the definition of $\mathcal{P}_j$ ($= \tau_j$) (see (\ref{eq:2})), we obtain
	\[ \begin{split}
	 \tau_j(\mathbb{F}_q \langle a \rangle_{D_{n}}) = \left\{ 
		\begin{pmatrix}
			P(\alpha_j) & 0 \\ 0 & P(\alpha_j^{-1})
		\end{pmatrix} 
		\mathlarger \mid P(x) \in \mathbb{F}_q[x]
	 \right\} \subset
	 \left\{ 
	 		\begin{pmatrix}
	 			M(\alpha_j) & 0 \\ 0 & Q(\alpha_j^{-1})
	 		\end{pmatrix}  
	 		\mathlarger \mid M(x), Q(x) \in \mathbb{F}_q[x]
	 \right\} \subset \\
	 \subset \tau_j(\mathbb{F}_q \langle a \rangle_{D_{n}}).
	 \end{split} \]
	 Note that the last inclusion is obtained by claim 3) of Theorem \ref{theor:3} (see (\ref{eq:teor3:1})). Indeed, $\tau_j \epsilon_j = \mathrm{id}_{A_j}$ and
	 \[ \epsilon_j \begin{pmatrix}
	 M(\alpha_j) & 0 \\ 0 & Q(\alpha_j^{-1})
	 \end{pmatrix} \in \mathbb{F}_q\langle a \rangle_{D_{n}}. \]
	 So, using definition (\ref{eq:a1}), we see that
	 \[ \mathcal{P}_j(\mathbb{F}_q \langle a \rangle_{D_{n}}) \cap I_j(0, 1) = 
	 	\left\{ 
	 	\begin{pmatrix}
	 	P(\alpha_j) & 0 \\ 0 & 0
	 	\end{pmatrix}  \mathlarger \mid P(x) \in \mathbb{F}_q[x]
	 	\right\},
	 \] \[	 	
	 	\mathcal{P}_j(\mathbb{F}_q \langle a \rangle_{D_{n}}) \cap I_j(1, 0) = 
	 	\left\{ 
	 	\begin{pmatrix}
	 	0 & 0 \\ 0 & Q(\alpha_j^{-1})
	 	\end{pmatrix}  \mathlarger \mid Q(x) \in \mathbb{F}_q[x]
	 	\right\}, \quad
	 \]
	 and $ \mathcal{P}_j(\mathbb{F}_q \langle a \rangle_{D_{n}}) \cap I_j(x, y) = \{ 0 \} $ for any $x, y \in \mathbb{F}_q[\alpha_j]$ such that $x \neq 0$, $y \neq 0$. 
	 Hence using (\ref{eq:Ij}), we see that
	 \[ A_j \left( \mathcal{P}_j(\mathbb{F}_q \langle a \rangle_{D_{n}}) \cap I_j(0, 1) \right) = I_j(0, 1), \quad A_j \left( \mathcal{P}_j(\mathbb{F}_q \langle a \rangle_{D_{n}}) \cap I_j(1, 0) \right) = I_j(1, 0).  \] 
	 This concludes the proof of the theorem.
\end{proof}

\begin{corollary} \label{col:mind}
	Corollary \ref{col:matr} implies that a generating matrix of a code $C$ can be represented as
	\[
	\left( \begin{array}{c|c}
	G_{I, 1} & 0 \\
	0 & G_{I, 2} \\
	\cline{1-2}
	G_{I, 3} & G_{I, 4}
	\end{array}	\right).
	\]
	Note that theorems \ref{theor:8}, \ref{theor:9} and corollary \ref{col:matr} imply that a generating matrix of $C_{ext}$ can be represented as
	\[
	\left( \begin{array}{c|c}
	G_{I, 1} & 0 \\
	0 & G_{I, 2} \\
	\cline{1-2}
	G_{I, 3} & 0 \\
	0 &  G_{I, 4}
	\end{array}	\right),
	\]
	and a generating matrix of $C_{int}$ can be represented as
	\[
	\left( \begin{array}{c|c}
	G_{I, 1} & 0 \\
	0 & G_{I, 2} 
	\end{array}	\right).
	\]
\end{corollary}

\begin{remark}
	Let $C_n = \langle h \mid h^n \rangle$ be the cyclic group of order $n$. As is well--known, $\mathbb{F}_q C_n \simeq {\mathbb{F}_q[x]}/{(x^n - 1)}$. It follows that, any ideal $I \subset \mathbb{F}_qC_n$ is of the form $I = (\mathbb{F}_qC_n) g(h)$, where $g(x) \in \mathbb{F}_q[x]$ and $g(x) \mid x^n - 1$. Note that $I \subset \mathbb{F}_q C_n$ is also called a cyclic code of length $n$, and $g(x)$ is called the generating polynomial of $I$ (see \cite{Lint99}, chapter 6).  
\end{remark}

Let $C$ be a dihedral code. The following theorem provides generating polynomials of $C_{int}$ and $C_{ext}$.

\begin{theorem} \label{theor:pol}
	Let $\gcd(q, n) = 1$. Consider (\ref{eq:form})--(\ref{eq:15}). Let $C \subset \mathbb{F}_qD_{n}$ be a dihedral code such that
	\[	\mathcal{P}(C) = \bigoplus_{j=1}^{r+s} B_j, \quad
	B_j = \begin{cases}
	A_j, & j \in K_1 \\
	I_j(0, 1) & j \in K_2 \\
	I_j(1, 0), & j \in K_3 \\		
	0, & j \notin K_1 \cup K_2	\cup K_3
	\end{cases},
	\]
	where $K_1 \subset \{ 1, \dots, r+s \}$ and $K_2, K_3 \subset \{ r+1, \dots, r+s \}$ are some pairwise disjoint sets. Let
	\[ g(x) = \left( \prod_{\substack{j \in K_1 \cup K_2}} f_j(x) \right) \left(
	\prod_{\substack{r + 1 \leq j \leq r+s \\ j \in K_1 \cup K_3}} f^*_j(x) \right) \]
	Then $C = (\mathbb{F}_qD_{n}) f(a)$, where $f(x) = (x^n - 1) / g(x)$. 
\end{theorem}
\begin{proof}
	Using Theorem \ref{theor:matr} and Lemmas \ref{lem:7}--\ref{lem:8}, we see that the code $C$ can be represented as
	\[ C = \left\{
		 \sum_{\substack{1 \leq j \leq r \\ j \in K_1}} t_j e_{f_j}(a) +
		 \sum_{\substack{r + 1 \leq j \leq r+s \\ j \in K_1 \cup K_2}} t_j e_{f_j}(a) +
		 \sum_{\substack{r + 1 \leq j \leq r+s \\ j \in K_1 \cup K_3}} t_j e_{f^*_j}(a)
		 \mathlarger \; \mid \; t_j \in \mathbb{F}_qD_{n}
	\right\}.
	\]
	Since
	$ e_{f_i}(a) e_{f_j}(a) = e_{f^*_i}(a) e_{f_j}(a) = e_{f^*_i}(a) e_{f^*_j}(a) = 0 $
	for $i \neq j$ (see (\ref{eq:ef})), it follows that
	\[ C = (\mathbb{F}_qD_{n}) t(a), \quad 
		t(a) =   \sum_{\substack{1 \leq j \leq r \\ j \in K_1}} e_{f_j}(a) +
		\sum_{\substack{r + 1 \leq j \leq r+s \\ j \in K_1 \cup K_2}} e_{f_j}(a) +
		\sum_{\substack{r + 1 \leq j \leq r+s \\ j \in K_1 \cup K_3}} e_{f^*_j}(a).
	\]
	Using Lemma 1 of \cite{VedDeu20}, we get
	\[ t(a) = e_g(a), \quad 
	g(x) = \left( \prod_{\substack{1 \leq j \leq r \\ j \in K_1}} f_j(x) \right) \left( 
	\prod_{\substack{r + 1 \leq j \leq r+s \\ j \in K_1 \cup K_2}} f_j(x) 
	\right) \left(
	\prod_{\substack{r + 1 \leq j \leq r+s \\ j \in K_1 \cup K_3}} f^*_j(x) \right). \]
	Definition (\ref{eq:ef}) implies that $(\mathbb{F}_q\langle a \rangle_{D_{n}})e_g(a) = (\mathbb{F}_q\langle a \rangle_{D_{n}}) f(a)$; this concludes the proof of the theorem.
\end{proof}

\begin{remark}
	Let $p(x)$ be a monic divisor of $x^n - 1$, and let
	\[ \begin{split}
	K_1 = \left\{ j = 1, \dots, r+s \mid \left(f_j(x) \nmid p(x) \right) \land \left( f^*_j(x) \nmid p(x) \right) \right\}, \\
	 K_2 = \left\{ j = r+1, \dots, r+s \mid \left(f_j(x) \nmid p(x) \right) \land \left( f^*_j(x) \mid p(x) \right) \right\}, \\
	K_3 = \left\{ j = r+1, \dots, r+s \mid \left(f_j(x) \mid p(x) \right) \land \left( f^*_j(x) \nmid p(x) \right) \right\}; 
	\end{split}
	\]
	then we see that $p(x)$ is equal to $f(x)$ from Theorem \ref{theor:9}. Therefore, there is a one--to--one correspondence between the cyclic codes of length $n$, which are generated by a monic divisor of $x^n-1$, and pairwise disjoint sets $K_1 \; (\subset \{ 1, \dots, r+s \})$ and $K_2, K_3 \; (\subset \{ r+1, \dots, r+s \})$.
\end{remark}

\section{Code parameters and some examples} \label{sec:est}
In this section some approximate estimates of code parameters are provided. In addition, some illustrative examples of the dihedral codes are given.

Let $C \subset \mathbb{F}_qG$ be a group code. Recall that the main code parameters are the length $|G|$, the dimension $k = \dim_{\mathbb{F}_q}(C)$ and the minimum distance $d(C) = \min_{c \in C, c \neq 0} w(c)$. If $G=D_{n}$, then the length and the dimension are pretty easy to find. Indeed, $|D_{n}| = 2n$, and if $C$ is given by (\ref{eq:form})--(\ref{eq:15}), then using (\ref{eq:2}), (\ref{eq:dim}) and (\ref{eq:a1}), we obtain
\begin{equation}  \label{eq:30}
\begin{split}
	\dim_{\mathbb{F}_q}(C) = \sum_{j=1}^{r+s} \dim_{\mathbb{F}_q}(B_j) = 
	2\left(\sum_{\substack{1 \leq j \leq r \\ j \in J_1}} \deg(f_j) \right) +
	\left( \sum_{\substack{1 \leq j \leq r \\ j \in J_2 \cup J_3 \cup J_4}} \deg(f_j) \right)
	+ \\ + 4 \left( \sum_{\substack{r+1 \leq j \leq r+s \\ j \in J_1}} \deg(f_j) \right) + 
	2 \left(\sum_{\substack{r+1 \leq j \leq r+s \\ j \in J_2 \cup J_3 \cup J_4}} \deg(f_j) \right).
\end{split}
\end{equation} 
In addition, Formula (\ref{eq:30}) and Corollary \ref{col:2} imply that any self--dual code in $\mathbb{F}_qD_{n}$ has dimension $k = n$.

Using (\ref{eq:chain}), we also obtain the bound $d(C_{ext}) \leq d(C) \leq d(C_{int})$ on the minimum distance. In addition, since
\[ C_{int} =\mathbb{F}_qD_{n}(\mathbb{F}_q \langle a \rangle_{D_{n}} \cap C), \quad C_{ext} = \mathbb{F}_qD_{n}\mathtt{pr}_a(C), \]
it follows that $d(C_{ext}) = d(\mathbb{F}_q \langle a \rangle_{D_{n}} \cap C)$, $d(C_{ext}) = d(\mathtt{pr}_a(C))$.
Note that for some dihedral codes the lower bound can be specified.
\begin{theorem} \label{theor:11}
Let $\gcd(q, n) = 1$, let $C \subset \mathbb{F}_qD_{n}$ be a dihedral code such that
\[
	\mathcal{P}(C) = \bigoplus_{j=1}^{r+s} B_j, \quad
	B_j = \begin{cases}
	A_j, & j \in J_1 \\
	I_j(0, 1), & j \in J_2 \\
	I_j(1, 0), & j \in J_3 \\
	I_j(-x_j, 1) & j \in J_4 \\			
	0, & j \notin J_1 \cup J_2	\cup J_3 \cup J_4
	\end{cases},
\]
\[J_1 = J_2 \cap \{r+1, \dots, r+s\} = J_3 \cap \{r+1, \dots, r+s\} = \varnothing;
\] then $d(C) \geq 2d(C_{ext})$.	
\end{theorem}
\begin{proof} Theorem \ref{theor:matr} and corollary \ref{col:matr} imply that a generator matrix of $C$ can be represented as
\[ \begin{pmatrix}
G_{C, 1} & SG_{C, 1}
\end{pmatrix}, \]
where $G_{C, 1}$ is a generating matrix of a cyclic code, and $S$ is an invertible matrix. Corollary \ref{col:mind} implies that a generating matrix of $C_{ext}$ is of the form
\[ \begin{pmatrix}
G_{C, 1} & 0 \\ 0 & SG_{C, 1}
\end{pmatrix}, \]
this concludes the proof of the theorem.
\end{proof}

In fact, for some codes the equality $d(C) = 2d(C_{ext})$ holds. Below, we give a sufficient condition.
\begin{theorem} \label{theor:12}
	Let $\gcd(q, n) = 1$, let $K \subset \{1, \dots, r+s\}$ be a nonempty set, and let $C \subset \mathbb{F}_qD_{n}$ be a code such that
	\[
	\mathcal{P}(C) = \bigoplus_{j=1}^{r+s} B_j, \quad
	B_j = \begin{cases}
	I_j(1, 0), & (j \in K) \land \left( 1 \leq j \leq \zeta(n) \right) \\
	I_j(\alpha_j + \alpha_j^{-1}, 2), & (j \in K) \land \left( \zeta(n) + 1 \leq j \leq r \right) \\
	I_j(-1, 1), & (j \in K) \land \left(  r+1 \leq j \leq r+s \right) \\
	0, & j \notin K
	\end{cases};
	\] then $d(C) = 2d(C_{ext})$.	
\end{theorem}
\begin{proof}
	Note that if $2 \nmid q$, then $I_j(\alpha_j + \alpha_j^{-1}, 2) = I_j((\alpha_j + \alpha_j^{-1})/2, 1)$; and if $2 \mid q$, then $I_j(\alpha_j + \alpha_j^{-1}, 2) = I_j(1, 0)$ (see Remark \ref{remark:4}). Theorem \ref{theor:11} implies that $d(C) \leq 2d(C_{ext})$. It remains to prove that there exists an element of weight $2d(C_{ext})$ in $C$. Theorem \ref{theor:matr} implies that $C$ has a basis of the form
	$T = \bigcup_{j=1}^{r+s} \mathcal{B}(B_j)$.
	Using Lemmas \ref{lem:7}--\ref{lem:8} and $be_{f_j}(a) =  e_{f_j}(a^{-1})b = e_{f_j^*}(a)b$, we obtain 
	\begin{enumerate}
		\item[1)] for $1 \leq j \leq \zeta(n)$:
		\[ \mathcal{B}(B_j) = \mathcal{B}(I_j(1, 0)) = \{ e_{f_j}(a) + be_{f_j}(a) \} = \{ e_{f_j}(a) + e_{f_j}(a)b \}; \]
		\item[2)] for $\zeta(n) + 1 \leq j \leq r$:
		\[ \begin{split}
			 \mathcal{B}(B_j) = \mathcal{B}(I_j(\alpha_j + \alpha_j^{-1}, 2)) =
			\{ a^i (1+b) e_{f_j}(a) \mid i=0,\dots,\deg(f_j)-1 \} = \\
			= \{ a^i e_{f_j}(a) + a^{i} e_{f_j}(a)b \mid i=0,\dots,\deg(f_j)-1 \};
			\end{split} 
		\]
		 \item[3)] for $r + 1 \leq j \leq r+s$:
		 \[ \begin{split}
		 	\mathcal{B}(B_j) = \mathcal{B}(-1, 1) = \{ a^i b^k (e_{f_j}(a) + be_{f^*_j}(a)) \mid i = 0, \dots, \deg(f_j)-1, \; k=0, 1 \} = \\
		 	\{ a^i e_{f_j}(a) + a^i e_{f_j}(a) b, \; a^i e_{f^*_j}(a) + a^i e_{f^*_j}(a) b \mid i = 0, \dots, \deg(f_j)-1 \}.
		 	\end{split}
		 \]
	\end{enumerate} 
	Hence any element of $C$ is of the form $P(a) + P(a)b$, where $P(a) \in C_{ext}$. Let $P(a)$ be an element of $C_{ext}$ such that $w(P(a)) = d(C_{ext})$, it follows that $w(P(a) + P(a)b) = 2d(C_{ext})$. The theorem is proved.
\end{proof}

\textbf{Example 1.} Let $q=5$ and let $n=4$. Consider $\mathbb{F}_qD_{n}$.  We have
\[ x^5-1 = \left( (x-1)(x+1) \right) \left( (x+2)(x-2) \right),  \]
here the polynomials $x-1$, $x+1$ are auto--reciprocal, and the polynomials $x+2$, $x-2$ are a pair of non-auto-reciprocal polynomials, i.e. $r = 2$, $s = 1$ (see (\ref{eq:dec})). Consider the code $C \subset \mathbb{F}_qD_{n}$, given by
\[ C = \mathcal{P}^{-1}(I_1(1, 0) \oplus I_2(1, 0) \oplus I_3(-2, 1)) \]
(see (\ref{eq:a1}) and Theorem \ref{theor:codes}). Clearly, its length is $2n = 8$, its dimension is $4$ (see (\ref{eq:30})), and by direct calculations we see that its minimum distance is $4$. In addition, $C_{out} = \mathbb{F}_qD_{n}$ and $d(C_{out}) = 1$. So, $C$ is a $[8, 4, 4]_5$--code, such that $d(C) > 2d(C_{out})$.

\textbf{Example 2.} In \cite{Mil19} a $[18, 2, 15]_{11}$--code in the semisimple algebra $\mathbb{F}_{11}D_{18}$ was considered. It was proved that this code is not combinatorially equivalent to any abelian code, and the weight of this code is the same as that of the best--known code of same dimension.

\textbf{Example 3.} 
Let $q=2$, $n=15$. Consider $\mathbb{F}_qD_{n}$. We have
\[ x^n-1 = \left( (x + 1)(x^2 + x + 1)(x^4 + x^3 + x^2 + x + 1) \right) 
\left( (x^4 + x + 1) (x^4 + x^3 + 1) \right) \]
(see (\ref{eq:dec})). Note that $r=3$, $s = 1$ in this case. Consider the code $C_1$, given by
\[ 
	C_1=\mathcal{P}^{-1}\left( I_1(1, 0) \oplus I_2(1, 0) \oplus I_3(1, 0) \oplus I_4(1, 1) \right) 
\] 
(see (\ref{eq:a1}) and Theorem \ref{theor:codes}). Clearly, $C_{out} = \mathbb{F}_qD_{n}$ and $d(C_{out}) = 1$.  Corollary \ref{col:2} implies that $C$ is a self--dual code, and Theorem \ref{theor:12} implies that its minimum distance is $2$. Hence $C_1$ is a binary self--dual dihedral $[30, 15, 2]$--code. Let
\[
C_2 = \mathcal{P}^{-1} \left(I_1(1, 0) \oplus I_2(1, 0) \oplus I_3(1, 0) \oplus I_4(1, 0) \right). 
\]
This code is also self--dual and $\dim_{\mathbb{F}_q}(C_2) = 15$. Using Theorem \ref{theor:matr}, we can build its generating matrix. And by direct calculation, we conclude that $d(C_2) = 6$.

Note that in \cite{Cao} the self--dual binary dihedral codes in $F_2D_{8m}$, i.e. when $\gcd(q, n) \neq 1$, were described. In particular, self--dual binary dihedral $[48, 24, 12]$--,  $[56, 28, 12]$--codes were built.

\textbf{Example 4.} In \cite{VedDeu18} an approach to obtain non--induced dihedral codes $C$, for which $\mathtt{pr}_a(C)$ is a Reed--Solomon code, was considered. Note that if $\mathtt{pr}_a(C)$ can correct up to $t$ errors, then the decoder of $\mathtt{pr}_a(C)$ can be used to decode $C \; (\subset C_{ext} = \mathbb{F}_qD_{n}\mathtt{pr}_a(C))$ if no more than $t$ errors occur in the channel. Below we give an example from \cite{VedDeu18} of such a code. We have
\[ x^{10} - 1 = \left( (x-1)(x+1)  \right) \left( [(x-2)(x-6)][(x-3)(x-4)][(x-7)(x-8)][(x-9)(x-5)] \right) \]
(see (\ref{eq:dec})). Here first two factors are auto--reciprocal, and the rest of the factor are non--auto--reciprocal, i.e. $r=2$, $s = 4$. Let
\[ C = \mathcal{P}^{-1} \left( A_1 \oplus 0 \oplus I_3(1, -1) \oplus A_4 \oplus 0 \oplus 0 \right) \]
(see (\ref{eq:a1}) and Theorem \ref{theor:codes}). Clearly, $\mathtt{pr}_a(C)$ is a $[10, 5, 6]$--code, and by direct calculations, we see that $C$ is $[20, 8, 8]$--code. Hence $\mathtt{pr}_a(C)$ can correct up to $2$ errors, and $C$ can correct up to $3$ errors. In addition, if no more than $2$ errors occur in the channel, then the decoder of $\mathtt{pr}_a(C)$ can be used to decode $C$.

\bibliographystyle{unsrt}
%\setcitestyle{authoryear,open={(},close={)}}
\bibliography{references}  %%% Remove comment to use the external .bib file (using bibtex).
\end{document}